\documentclass{compositio}

\usepackage{amsmath,amsfonts,hyperref}

\usepackage{amsthm}
\usepackage{amssymb}
\usepackage[utf8]{inputenc}
\usepackage[nottoc,notlof,notlot]{tocbibind}
\usepackage{mathtools}
\usepackage{enumitem}
\usepackage{tikz-cd}
\usepackage{multirow}
\usepackage{mathrsfs}
\usepackage{array}
\usepackage{pgf,tikz}
\usetikzlibrary{arrows}
\usetikzlibrary[patterns]
\usepackage{soul}

\newcommand{\PP}{\mathbb{P}}

\newcommand{\ZZ}{\mathbb{Z}}
\newcommand{\RR}{\mathbb{R}}
\newcommand{\CC}{\mathbb{C}}
\newcommand{\CP}{\mathbb{C}P}
\newcommand{\RP}{\mathbb{R}P}

\newcommand{\CCC}{\mathscr{C}}

\renewcommand{\L}{\mathcal{L}}

\newcommand{\X}{\mathcal{X}}

\newcommand{\C}{\mathcal{C}}
\newcommand{\Y}{\mathcal{Y}}

\renewcommand{\H}{\mathcal{H}}
\renewcommand{\O}{\mathcal{O}}

\renewcommand{\epsilon}{\varepsilon}

\newtheorem{theo}{Theorem}[section]
\newtheorem*{theom}{Theorem}
\newtheorem{prop}[theo]{Proposition}
\newtheorem*{propo}{Proposition}
\newtheorem{coro}[theo]{Corollary}
\newtheorem{lem}[theo]{Lemma}

\theoremstyle{definition}
\newtheorem{defi}[theo]{Definition}
\newtheorem{prob}[theo]{Problem}

\theoremstyle{remark}
\newtheorem{remark}[theo]{Remark}
\newenvironment{rem}[1]{
    \begin{remark}#1}{
    \xqed{\blacklozenge}\end{remark}
}

\theoremstyle{remark}
\newtheorem{example}[theo]{Example}
\newenvironment{expl}[1]{
    \begin{example}#1}{
    \xqed{\lozenge}\end{example}
}

\newcommand{\xqed}[1]{
    \leavevmode\unskip\penalty9999 \hbox{}\nobreak\hfill
    \quad\hbox{\ensuremath{#1}}}

\DeclarePairedDelimiter{\floor}{\lfloor}{\rfloor}

\title{Bitangents of real algebraic curves: signed count and constructions}

\author{Thomas Blomme}
\email{thomas.blomme@unige.ch}
\address{Universit\'e de Neuch\^atel, avenue \'Emile Argan 11, Neuch\^atel 2000, Suisse}

\author{Erwan Brugall\'e}
\email{erwan.brugalle@math.cnrs.fr}
\address{Nantes Universit\'e, Laboratoire de
  Math\'ematiques Jean Leray, 2 rue de la Houssini\`ere, F-44322 Nantes Cedex 3,
France}

\author{Cristhian Garay}
\email{cristhian.garay@cimat.mx}
\address{Centro de Investigaci\'on en Matem\'aticas, A.C. (CIMAT),
Jalisco S/N, Col. Valenciana CP. 36023 Guanajuato, Gto, M\'exico.}

 
\begin{document}
 

\begin{abstract}
  We study real bitangents of real algebraic plane curves from two
  perspectives. We first show that there exists a signed count of such
  bitangents that only depends on the real topological type of the
  curve. From this follows that a generic real algebraic curve of
  even degree $d$ has at least $\frac{d(d-2)}{2}$ real
  bitangents. 
  Next we explain how to locate (real) bitangents of a (real)
  perturbation of a multiple (real) conic in $\CP^2$.
As main applications, we exhibit a real sextic with a total of $318$ real
bitangents and 6 complex ones, and  perform asymptotical
constructions that give the best, to our knowledge, number of real
bitangents of real algebraic plane curves of a given degree.
\end{abstract}

\maketitle

\tableofcontents

\setstcolor{red}

An algebraic curve in $\CP^2$ is said to be \emph{generic} if it is non-singular and if its dual curve has only nodes and cusps as its singularities. Given a real plane algebraic curve $C$ in $\CP^2$, i.e. defined by a polynomial with real coefficients, we define the \emph{real scheme} of $C$ as the topological type of the pair $(\RP^2,\RR C)$.

\section{Results}

\subsection{A brief historical overview}
Plücker proved in 1834  \cite{Plu34} that a generic complex algebraic plane curve of degree $d$ in  $\CC P^2$ has exactly

\begin{enumerate}
\item $i_\CC(d)=3d(d-2)$ inflection points;
\item  $t_\CC(d)=\frac{1}2 d(d-2)(d-3)(d+3)$ bitangents.
\end{enumerate}

For real curves the situation quickly becomes quite intricate regarding their  real  bitangents and real inflection points. Curves of degree 4 constitute the first non-trivial situation for bitangents, and Plücker noticed in \cite{Plu39} the following interesting features:

  \begin{itemize}[label=$\ast$]
  \item    there exists  plane quartics  whose $t_\CC(4)=28$ bitangents are all real;
\item  not all possible even values between $0$ and $28$ can be realized as the number of real bitangents of a general real quartic curve. 
  \end{itemize}

Based on these observations, Zeuthen  extensively studied in \cite{Zeu74}  real plane quartics and their real bitangents. He proved in particular that the number of real bitangents of such curves is determined by its real scheme, as indicated on Figure \ref{fig:quartic}, see also \cite[Section 7]{GrHa83}.

  \begin{figure}[!htb]
\centering
\begin{tabular}{ccccccccccc}
  \includegraphics[height=2cm, angle=0]{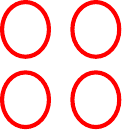}
  & \hspace{1ex} &
  \includegraphics[height=2cm, angle=0]{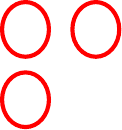}
  & \hspace{1ex} &
  \includegraphics[height=2cm, angle=0]{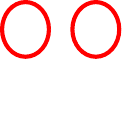}
  \\   28 real bitangents& & 16  real bitangents && 8  real bitangents
   \\\\ 
  \includegraphics[height=2cm, angle=0]{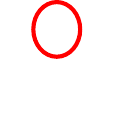}
  & \hspace{1ex} &
  \includegraphics[height=2cm, angle=0]{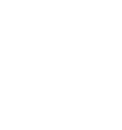}
  & \hspace{1ex} &
  \includegraphics[height=2cm, angle=0]{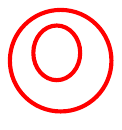}
  \\ 4  real bitangents && 4 real bitangents && 4  real bitangents  
\end{tabular}
\caption{Number of real bitangents to real quartic curves in terms of their real scheme.}\label{fig:quartic}
  \end{figure}
 
  Following Zeuthen's work, Klein proved   in \cite{Kle2}  that a generic  real plane algebraic  curve $C$ of degree $d$   satisfies the  remarkable formula
\begin{equation}\label{eq:klein}
i+2t_0=d(d-2),
\end{equation}
where $t_l$ is the number of real bitangents of $C$ that are tangent to $C$ at $l$ real points ($l=0$ or $l=2$), and  $i$ denotes the number of real inflection points of $C$. Note that even if the quantities $t_2,t_0,$ and $i$  depend on the curve $C$,  we do not indicate this dependency in order to lighten up notations. Klein's Formula \eqref{eq:klein} was later generalized by Schuh \cite{Schuh} to singular curves. Klein's proof  has eventually not been considered as entirely rigorous, and  several complete proofs (for modern standards) of the Klein Formula have been given since the 80's, see \cite{Viro-Euler,Wal96,Ron1}. 

As  a consequence of Klein's Formula \eqref{eq:klein},  one obtains that
\begin{itemize}[label=$\ast$]
\item $i\le d(d-2)$, i.e. the number of real inflection points is no more than  one third of the number of complex inflection points;
\item  $t_0\le \frac{d(d-2)}2$, i.e. there can be just a few real bitangents of  type $t_0$.
\end{itemize}
Furthermore, these bounds are easily seen to be sharp, provided that $d$ is even for the second inequality\footnote{Exercise:  check  that $i\ge 3$ for odd degrees.}.

Quartic plane curves and their bitangents have been popular subjects since at least two centuries, due to their connections with other important topics in algebraic geometry (del Pezzo surfaces, theta characteristics of algebraic curves, etc.) Subsequently, an abundant literature has been devoted to them from various (semi-)field perspectives. We refer  for example to \cite{Edge94,CapSer03,PSV11,BLMPR16,ChaJir17,LeeLen18,LenMar20,LarVog21,CueMar23,MPS23,MPS23b,GeiPan21} for a glimpse of such works from the last decades.

Studies of real bitangents of real plane curves of higher degrees are much more sporadic. Coolidge showed in \cite{Coo31} that any non-trivial polynomial relation among the quantities $i, t_0, t_2$ and $d$ is a consequence of Plücker and Klein relations. A signed count of  real bitangents of a real plane  sextic, that eventually  depends only  on  its   real scheme,  was obtained by   Kharlamov and R\u{a}sdeaconu in \cite{KhaRas13} as a special case of enumeration of rational curves in real $K3$-surfaces. Still in degree 6,  Nidhi,  Kummer, Plaumann, Namin, and Sturmfels made in \cite{NKPNS19} a thorough experimental study of several aspects of  real plane sextics. In particular they could realize several even numbers between 12 and 306 as the number of real bitangents of a real plane sextic (note that $t_\CC(6)=324$).

These are just about all the general results we are aware of regarding real bitangents of real plane  algebraic  curves.      To our knowledge, the following problem is still widely open as soon as $d\ge 5$.

\begin{prob}
Which integer numbers can occur as the number of real bitangents of a generic real plane algebraic curve of degree $d$? In particular what is the maximal number $t_{max}(d)$ of real bitangents that  real   algebraic curve of degree $d$ can have? For which numbers $d$ does the equality  $t_{max}(d)=t_\CC(d)$ hold?
\end{prob}

{\bf From now on, we focus our attention to  algebraic plane curves  of even degree.}
In this paper we provide  two contributions to the study of real bitangents of real plane algebraic curves, namely:
\begin{itemize}
\item[$(1)$] a signed enumeration that eventually only depends on the degree and the real scheme of the curve; this enumeration encompasses the two previous enumerations by Zeuthen and  Kharlamov-R\u{a}sdeaconu,
\item[$(2)$] the construction of real algebraic plane curves with a controlled topology and position with respect to its real bitangents; in particular we show that $t_{max}(6)\ge 318$,  disproving \cite[Conjecture 5.8]{NKPNS19}.
\end{itemize}

\subsection{Signed count}\label{sec:signed}
Our first main result is the existence of a signed count of real bitangents of a real plane algebraic curve $C$ of even degree that only depends on its degree and real scheme. Let $\RR P^2\setminus \RR C=\RR P^2_+\sqcup \RR P^2_-$ be a partition of $\RR P^2\setminus \RR C$ satisfying $\partial \RR P^2_\pm=\RR C$. Such a partition exists since $C$ has even degree, and it is unique since $\RP^2$ is connected. 

Given  a real line $L$ tangent to $C$ at a real point $p$, we define  $\varepsilon_L(p)\in\{\pm 1\}$   by the condition that
 $\RR L$ is locally contained in $\RP^2_{\varepsilon_L(p)}$ in a neighborhood of $p$. 
The sign $\varepsilon(L)$ of a real bitangent $L$ of $C$ is then defined as follows:
\begin{enumerate}
  \item if $L$ is tangent to $C$ at two complex conjugated points of
  $C$, then $\varepsilon(L)=+1$;

\item if $L$ is tangent to $C$ at two real points $p_1,p_2\in\RR P^2$,
  then
  $\varepsilon(L)=\varepsilon_L(p_1)\varepsilon_L(p_2)$.
\end{enumerate}

We denote by $t_s(C)$ the corresponding signed count of real bitangents to the real plane algebraic curve $C$.
Recall that any connected component $o$ of $\RR C$, called an \emph{oval}, divides $\RP^2$
into two connected components, one of which, called the \emph{interior} of $o$, is homeomorphic to a disk. Such oval is said to be \emph{even} or \emph{odd}  depending on whether it is contained in the interior of an even or odd number of other ovals of $\RR C$.
The distinction between even and odd ovals of a real plane algebraic curve of even degree goes back to the work of Ragsdale \cite{Rag}, and played an important role in the development of real algebraic geometry since then. Interestingly, this distinction is also relevant when considering real bitangents. Section \ref{Sect:ProofFirstMainTheorem} is devoted to the proof of next theorem.

\begin{theo}\label{thm:signed}
  Let $C$ be a generic real algebraic curve of even degree $d$ in  $\CC P^2$ with $p$ even ovals and $n$ odd ovals. Then one has
  \[
  t_s(C)= 2(p-n)(p-n-1)+\frac{d(d-2)}2.
  \]
\end{theo}

Note that $t_s(C)\ge \frac{d(d-2)}2$. Since by definition $t_s(C)$ bounds from below the number of real bitangents of $C$, we get the following result. 
\begin{coro}\label{cor:lowerbound}
A generic real plane algebraic curve of even degree $d$ has at least $\frac{d(d-2)}2$ real bitangents.
\end{coro}

As mentioned above, Theorem \ref{thm:signed} and Corollary \ref{cor:lowerbound} recover results by Zeuthen and Kharlamov-R\u{a}sdeaconu in the case of quartics and sextics, respectively.

\begin{expl}[The case $d=4$]
In this case  Theorem \ref{thm:signed} gives
  \[
  t_s(C)=2(p-n)(p-n-1)+4.
  \]
  Note that $n=0$ for a real quartic curve unless $n=p=1$,   see Figure \ref{fig:quartic}.  
Let $L$ be a real bitangent to a generic real quartic $C$. By Bézout theorem,
the real part $\RR L$  is contained in the  topological closure of  the non-orientable component of
  $\RR P^2\setminus\RR C$, in particular $\varepsilon(L)=+1$. Hence
 $t_s(C)$ is the actual number of real bitangents of $C$, and
 Theorem \ref{thm:signed} specializes to  
Zeuthen's result \cite{Zeu74}.
\end{expl}

\begin{expl}[The case $d=6$]
  In this case, Theorem \ref{thm:signed} gives
  \[
  t_s(C)=2(p-n)(p-n-1)+12.
  \]
  Taking the double covering  $\pi:\CC S\to \CC P^2$ of $\CC P^2$ branched along a real sextic
$C\subset \CP^2$, endowed with any of its two real structures, we  obtain that
the real algebraic surface $S$ contains
\[
t_s=\frac 12 \chi(\RR S)\left( \chi(\RR S)-2 \right)+\chi(\CC S)
\]
  real rational curves realizing the class $\pi^*(\O_{\CP^2}(1))$, counted with
  Welschinger sign. This is \cite[Theorem A, Section 4.1]{KhaRas13}
  applied to 
  $g=2$.
\end{expl}

The invariant $t_s$ is  in general far from being a sharp lower bound, in the sense that a curve $C$ usually has real bitangents of
different signs. The next result is obtained using the construction method detailed below and in
Section \ref{sec:conic}.
\begin{prop}\label{prop:optimality}
For the following real schemes, there exists a generic real plane algebraic curve $C$ of degree $2k$  with exactly $t_s(C)$ real bitangents:
\begin{enumerate}
  \item $\RR C$ is  a disjoint union
    of $p\le \frac{3k(k-1)}{2}+1$ empty ovals;
  \item $k=2l+1$ and  $\RR C$  a disjoint union
   of one oval enclosing $n$ empty ovals (i.e. $\RR C$ has exactly $n+1$
  connected components), with  $n\le \frac{3k(k-1)}{2}$.  
  \end{enumerate}
\end{prop}

\begin{remark}
  It follows form Comessati Inequality, see for example  \cite{Mang17}, that a non-singular real algebraic  curve of degree $2k$ with either $p$ or $n+1$ ovals satisfies  $p, n\le \frac{3k(k-1)}{2}+1$. Our proof of Proposition
  \ref{prop:optimality}  does not cover the case $k=2l+1$ and
  $n=\frac{3k(k-1)}{2}+1$.  Assuming that any pair   of  ovals each one outside the other
have at
  least one common bitangent, a necessary condition for all the
  bitangents to have the same sign is for the real scheme to consist
  of a union of empty ovals contained in the interior of a nest of
  ovals, see Figure \ref{fig:sharp}. It   would be interesting to
  investigate whether there exists a generic real plane algebraic
  curve $C$ with exactly $t_s(C)$ real bitangents in that case.  
 
\begin{figure}[h!]
\centering
\begin{tabular}{c} 
  \includegraphics[height=4cm, angle=0]{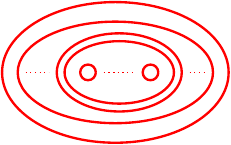}
\end{tabular}
\caption{Is the number $t_s$ sharp for this isotopy type?}
\label{fig:sharp}
\end{figure} 
\end{remark}

\subsection{Constructions}\label{sec:intro constr}
Our second main result is a method for constructing real algebraic curves that makes it possible to locate real bitangents. The precise statement requires a non-negligible amount of definitions, and will be given in Section \ref{sec:real bit}. For now, we confine ourselves to the following  informal formulation.

\begin{theom}\ref{thm:bit}
By perturbing a multiple plane conic, one can construct real algebraic plane curves of even degree with a controlled real scheme and position with  respect to its real bitangents.
\end{theom}

The idea of perturbing a multiple plane conic to  construct real algebraic plane curves of even degree with a controlled topology can be traced back to the work of Chevallier \cite{Che}, and has been successfully employed by different authors since then \cite{O2,Ore6,FinKha15,Bru11}.  The added value of our work is to show that this method also allows to keep track of the  real bitangents. The idea behind Theorem \ref{thm:bit} is quite simple: while it is  difficult to study the position of a plane curve with respect to a net of lines, there exist powerful  tools to study its position  with respect to a pencil of lines \cite{O1,O2,O3,Ore2,IS,Br3,Lop}. Perturbing a multiple of a conic $C_0\subset\CP^2$ amounts  to working in the normal bundle $\mathcal N_{C_0/\CP^2}$ rather than in $\CP^2$, and replaces the net of lines of this latter with the fibers of the bundle map
 $\mathcal N_{C_0/\CP^2}\to C_0$.

\medskip
Now we present  three applications of Theorem \ref{thm:bit}. We start by applying thoroughly this method to sextics.
 \begin{figure}[!htb]
\centering
\begin{tabular}{c|c|c}
 $h$ & $\langle p\rangle$ &$\langle a \sqcup 1  \langle b\rangle\rangle$

\\ \hline &&\\ \includegraphics[height=1cm, angle=0]{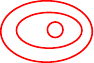}
  &
  \includegraphics[height=1cm, angle=0]{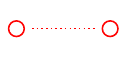}
  \put(-55,-5){$\underbrace{\qquad\qquad\quad
      }_{\mbox{$a$ ovals}}$}
  &
  \includegraphics[height=1cm, angle=0]{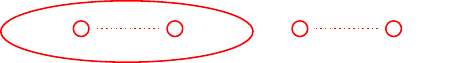}
  \put(-175,-5){$\underbrace{\qquad\qquad\quad
      }_{\mbox{$a$ ovals}}$}
  \put(-74,-5){$\underbrace{\qquad\qquad\quad
      }_{\mbox{$b$ ovals}}$} 
\\&& \\ &   $a\le 10$
& $a+b\le 10$ and   
$\left\{\begin{array}{l}
a=b\mod 8 \mbox{ if }a+b=10
\\ a=b\pm 1\mod 8 \mbox{ if }a+b=9
\end{array}\right.$
\end{tabular}
\caption{Real schemes realized by real plane sextics.}\label{fig:rc sextic}
  \end{figure}
  
 It is a result of Gudkov \cite{Gud}, see also \cite{DK,Mang17}, that the  real schemes realized by real plane sextics are exactly those depicted on Figure \ref{fig:rc sextic}. Given such real scheme  $\mathcal R$, we denote   by $ t_{max}^{\mathcal R}(6)$ the maximal number of real bitangents that  a non-singular real plane sextic with real scheme $\mathcal R$ may have.
 
\begin{theo}\label{thm:sextics}
  Let  $\mathcal R\ne \emptyset,h$ be a real scheme realized by a  real plane sextic with $p$ even ovals and $n$ odd ovals. Then
  \[
  t_{max}^{\mathcal R}(6)\ge 2(p-n)+298.
  \]
\end{theo}

\begin{rem}
  If $\mathcal R= \emptyset$, then Klein Formula \eqref{eq:klein} implies that a real  sextic with empty real part has exactly 12 real bitangents.
\end{rem}

The lower bound from Theorem \ref{thm:sextics}  is maximal when
$p=10$ and $n=0$, this is:
\[
t_{max}^{\langle 10\rangle}(6)\ge 318.
\]
Our Theorem \ref{thm:sextics} improves almost all lower bounds on
$ t_{max}^{\mathcal R}(6)$ from \cite{NKPNS19}, with the notable exception of the Gudkov real scheme $\langle 5 \sqcup 1  \langle 5\rangle\rangle$. In this case we obtain the lower bound $t_{max}^{\langle 5 \sqcup 1  \langle 5\rangle\rangle}(6)\ge300$, which  is smaller than the lower bound $t_{max}^{\langle 5 \sqcup 1  \langle 5\rangle\rangle}(6)\ge 306$ obtained also in \cite{NKPNS19}.
Note that Theorem \ref{thm:sextics} is the optimal result that can be achieved by perturbing a triple non singular conic. Hence there should remain some room left to improve Theorem \ref{thm:bit} to a more efficient systematic way of constructing real plane algebraic curves with a controlled position with respect to their real bitangents.

  In the case of maximal sextics, Theorem \ref{thm:sextics} gives the following lower bounds:
\[
t_{max}^{\langle 9 \sqcup 1  \langle 1\rangle\rangle}(6)\ge 316,
\qquad
t_{max}^{\langle 5 \sqcup 1  \langle 5\rangle\rangle}(6)\ge 300,
\qquad
t_{max}^{\langle 1 \sqcup 1  \langle 9\rangle\rangle}(6)\ge 284.
\]
Note that the two lower bounds on
$t_{max}^{\langle 9 \sqcup 1  \langle 1\rangle\rangle}(6)$ and
$t_{max}^{\langle 10\rangle}(6)$
 disprove \cite[Conjecture 5.8]{NKPNS19}.

\begin{rem}
One can refine Theorem \ref{thm:sextics} by considering $ t_{max}^{\mathcal U}(6)$ for  \emph{rigid isotopy classes} $\mathcal U$ of real plane sextics. By Nikulin's rigid isotopy classification of real sextics  \cite{Nik,DK,Mang17}, the data of such $\mathcal U$ is equivalent to the data of  a real scheme with the additional information whether $C\setminus \RR C$ is connected or not.
The refined statement for rigid isotopy classes is the same as Theorem \ref{thm:sextics} for all rigid isotopy classes except for $\emptyset$, $h$, and  the class of dividing curves realizing the real scheme $\langle 1 \sqcup 1  \langle 4\rangle\rangle$, which cannot be obtained by perturbing   a non-reduced conic.
\end{rem}

Next we focus on asymptotical constructions, which provide the best, that we are aware of up to now,  lower bounds regarding the number of real bitangents.  Note that these may not be the optimal constructions for a specific  degree,
as one sees already in next statement.
\begin{theo}\label{thm:asym1}
One has
\begin{enumerate}
\item $t_{max}(6)\ge 318 =t_\CC(6)-6$; 
  \item $t_{max}(d)\ge t_\CC(d)-\frac{d(d-2)(d-4)}{6}=t_{\CC}(2k)
 -\frac{4k(k-1)(k-2)}{3},
 \qquad \forall d=2k\ge 8.$
\end{enumerate}
\end{theo}
Part $(i)$ of Theorem \ref{thm:asym1} follows directly from
the lower bound $t_{max}^{\langle 10\rangle}(6)\ge 318$ obtained in
Theorem \ref{thm:sextics}.
We prove  part $(ii)$ of Theorem \ref{thm:asym1} by constructing  a
suitable family
of so-called \emph{Harnack curves} of even degree,
whose real
scheme is
depicted in Figure \ref{fig:harnack}. Real algebraic curves with this topological type  were originally
constructed by Harnack in \cite{Har}, and since then have appeared 
surprisingly  in several
mathematical contexts \cite{Mik11,MikRul01,KenOko06,Bru14b,BouCimTil23,BouCimTil23b}. 
\begin{figure}[h!]
\centering
\begin{tabular}{c} 
  \includegraphics[height=1cm, angle=0]{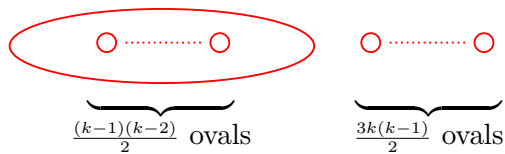}
  \put(-180,-5){$\smash{\underbrace{\qquad\qquad\quad
      }_{\frac{(k-1)(k-2)}{2}\mbox{ ovals}}}$}
  \put(-74,-5){$\smash{\underbrace{\qquad\qquad\quad
      }_{\frac{3k(k-1)}{2}\mbox{ ovals}}}$}
\\ \\ \\
\end{tabular}
\caption{Real scheme of Harnack curves of degree $2k$ in $\RP^2$.}
\label{fig:harnack}
\end{figure} 

As it can be seen from Theorems  \ref{thm:signed} and \ref{thm:sextics}, 
it seems relevant to
study the possible numbers of bitangents of real
plane algebraic curves within a fixed real scheme in addition to fixing a degree $d$.
For example, one may wonder if a poor real scheme imposes the existence of only few real
bitangents. The following result shows that this is not the case.
\begin{theo}\label{thm:asym2}
  There exists a sequence $(C_{d})_{d\ge 2}$ of real plane algebraic curves
of  even degree $d$ with a connected real part and with
$t_\CC(d) +O(d^3)$ real bitangents
  (i.e. all bitangents of $C_d$ are asymptotically real).
\end{theo}

\subsection{Outline of the paper}

We prove Theorem \ref{thm:signed}  in
Section \ref{sec:signedc}, after reviewing basic facts about
integration with respect to the Euler characteristic and Hilbert schemes.
 We state precisely and prove
Theorem \ref{thm:bit} in 
Section \ref{sec:conic}, and present  various applications in
Section \ref{sec:constr}. We conclude by outlining in
Section \ref{sec:persp}
some possible developments of the material presented in this paper.

\subsection{Acknowledgments}

We are grateful to Ilia Itenberg and Oleg Viro for their comments and advices.
Part of this work has been achieved during the visit of C.G. and
T.B. at Nantes
Université. C.G was 
 funded by  ECOS NORD 298995, CONACyT 282937, CONACyT I1200/381/2019. We
 thank  Laboratoire de Mathématiques Jean Leray
 for excellent working
 conditions. The authors also
 thanks the France 2030 framework program Centre Henri Lebesgue
 ANR-11-LABX-0020-01 for creating an attractive mathematical environment.

\section{Signed enumeration of real bitangents}\label{sec:signedc}
Our proof of Theorem \ref{thm:signed}
in inspired by the two works \cite{kool2011short} by
Kool, Shende and Thomas, and \cite{Viro-Euler} by Viro.
Namely, we consider the double covering
$\pi:S\to \CP^2$ ramified along a generic algebraic curve $C$ of
even degree $d$, and the line bundle $\pi^*\O_{\CP^2}(1)$ on
$S$. Let  $V\subset |\pi^*\O_{\CP^2}(1)|$ be the linear system\footnote{In fact
$V= |\pi^*\O_{\CP^2}(1)|$ if $d\ge 4$, but this will play no role.}
that parameterizes
curves $D$ in $S$ of the form $\pi^{-1}(L)$ with $L$ a line in
$\CP^2$. The
\emph{universal curve} $\CCC$ of $V$ is defined as
\[
\CCC=\Big\{(D,p)\ ,\ D\in V\mbox{ and } p\in D
\Big\}\subset V\times S,
\]
and comes with two tautological projections
\[
\begin{tikzcd}
		 & \CCC \arrow{dl}[swap]{\nu_1} \arrow{dr}{\nu_2} &  \\
		{V} &  & {S} \\
\end{tikzcd}
\]
The \emph{Hilbert square} $X^{[2]}$ of an algebraic variety $X$
 parameterizes the $0$-dimensional
closed sub-schemes of $X$
of length $2$. The \emph{relative Hilbert square} $\CCC^{[2]}$ is defined as
\[
\CCC^{[2]}=\left\{(D,p)\ ,\ D\in V\mbox{ and } p\in D^{[2]}
\right\}\subset V\times S^{[2]},
\]
and comes again with two tautological projections
\[
\begin{tikzcd}
		 & \CCC^{[2]} \arrow{dl}[swap]{\nu_1} \arrow{dr}{\nu_2} &  \\
		{V} &  & {S^{[2]}} \\
\end{tikzcd}
\]
When $C$ is a real curve, the surface $S$ carries two tautological
real structures, and any choice of one of them induces real structures
on $V$, $\CCC$,  $S^{[2]}$ and $\CCC^{[2]}$, which turn $\nu_1$ and $\nu_2$ into
real maps.
In particular these maps provide two expressions for the
topological Euler characteristic of $\RR\CCC^{[2]}$. Equating these two
expressions  proves Theorem \ref{thm:signed}.

\medskip
This section is devoted to detailing the 
strategy we have just outlined.
We start by recalling all needed facts about
integration with respect to Euler characteristic,
double coverings of $\CP^2$, and
Hilbert squares. 
Then we prove
Theorem \ref{thm:signed}. We end the section by indicating how our
technique to prove Theorem \ref{thm:signed} 
can be adapted to recover Plücker and Klein formulas in the special
case of even degree curves.

\subsection{Integration with respect to Euler characteristic}\label{sec:int euler}

We start by  recalling  principles of integration with respect to  Euler
characteristic, which
can be traced back at least to \cite{McPher74,Viro-Euler}.
We will  use extensively this technique throughout this section, and 
as we will see in many applications,
it provides a powerful tool to compute Euler characteristic of
algebraic varieties.

We will use the two following Euler characteristics: the topological one, usually denoted by $\chi$, and the one with compact support, usually denoted by $\chi_c$. It follows in particular from resolution of singularities, see for example
\cite[page 141]{Ful} or \cite[Lecture 1]{Peters10},
 that given three complex
algebraic varieties $X,Y,$ and $Z$, one has $\chi_c(X)=\chi(X)$ and
the additional following properties:
\begin{itemize}
\item Additivity: $\chi(X)=\chi(X\setminus Y)+\chi(Y)$  if
  $Y$ is a locally closed subvariety of $X$;
\item Multiplicativity: $\chi(X)=\chi(Y)\chi(Z)$ if
  $X$  is a locally trivial fibration
  with fiber $Z$ over the base $Y$ (in the topological sense).
\end{itemize}
This immediately implies next proposition.
\begin{prop}[(Integration with respect to Euler characteristic for
    complex algebraic varieties)]\label{prop:int}
  Let $f:X\to Y$ be an algebraic morphism between two complex algebraic
  varieties.
  Suppose that there exist two finite
  stratifications $\bigsqcup_{i=1}^n X_i$ and  $\bigsqcup_{i=1}^n Y_i$
into locally closed
  subvarieties  of $X$ and $Y$, respectively, 
   such that $f|_{X_i}:X_i\to Y_i$ is a locally trivial
  fibration with fiber $Z_i$ for $i=1,\ldots,n$. Then one has
  \[
  \chi(X)=\sum_{i=1}^n\chi(Y_i)\chi(Z_i).
  \]
\end{prop}

If we now assume that $X,Y$, and $Z$ are three real algebraic varieties, and
that we are given  semialgebraic subsets $U_{\RR X}$, $U_{\RR Y}$, and
$U_{\RR Z}$ of
$\RR X$, $\RR Y$, and $\RR Z$, respectively. Then one
has
\begin{itemize}
\item Additivity: $\chi_c(U_{\RR X})=\chi_c(U_{\RR X}\setminus
  U_{\RR Y})+\chi_c(U_{\RR Y})$  if
  $U_{\RR Y}$ is a   semialgebraic subset of $U_{\RR X}$;
\item Multiplicativity: $\chi_c(U_{\RR X})=\chi_c(U_{\RR Y})\chi_c(U_{\RR Z})$ if
  $U_{\RR X}$  is a locally trivial fibration
  with fiber $U_{\RR Z}$ over the base $U_{\RR Y}$.
\end{itemize}
Note that $\chi$ does not satisfy the additivity property in the real setting, as one can
see from the following situation:
\[
\chi(\RP^1)=0\ne 2=\chi(\RP^1\setminus \{p\})+ \chi(\{p\}).
\]

\begin{prop}[(Integration with respect to Euler characteristic for
    real algebraic varieties)]\label{prop:int r}
  Let $f:X\to Y$ be a real algebraic morphism between two real algebraic
  varieties.
  Suppose that there exist two finite
  stratifications $\bigsqcup_{i=1}^n U_{i}$ and  $\bigsqcup_{i=1}^n V_i$
into semialgebraic  subsets  of $\RR X$ and $\RR Y$, respectively, 
   such that $f|_{U_i}:U_i\to V_i$ is a locally trivial
  fibration with semialgebraic fibers $W_i$ for $i=1,\ldots,n$. Then one has
    \[
  \chi_c(\RR X)=\sum_{i=1}^n\chi_c(V_i)\chi_c(W_i).
  \]
\end{prop}

\begin{remark}
 One can  recast the additivity and multiplicativity of
 $\chi_c$ by saying that it provides a \emph{motivic invariant} on the
 \emph{Grothendieck rings} of complex algebraic varieties and real 
 semialgebraic varieties. This is not only a fancy way to formulate
 things, since there exists other motivic invariants, for example
the \emph{compactly supported $\mathbb A^1$-Euler characteristic} 
\cite{AMBOWZ22} for
fields of characteristic 0.
Proposition \ref{prop:int} has an immediate version for any motivic
invariant, replacing ``locally trivial fibration'' by ``trivial fibration''.
The fact that the multiplicativity property for $\chi$ and $\chi_c$ holds 
more generally for locally trivial
fibrations makes the computations easier, nevertheless nothing
prevent in principle to generalize the content of this section to, for example, 
compactly supported $\mathbb A^1$-Euler characteristic.
\end{remark}



\subsection{Double coverings of $\CP^2$}\label{sec:double}
        Let $C$ be a non-singular algebraic curve  of even degree
        $d=2k$ in
        $\CP^2$ given by the equation $P(x,y,z)=0$.
        The complex algebraic surface $S$ in the weighted projective plane
        $\CP^3(1,1,1,k)$ defined by the equation
        \[
        	t^2-P(x,y,z)=0
\]
is non-singular. 
Forgetting the $t$-coordinate 
exhibits $S$ as a double
covering $\pi:S\to \CP^2$ ramified along the curve $C$. Note that $S$
only depends on $C$ and not on the equation $P$:
given another equation $\lambda P(x,y,z)=0$ of $C$, with
$\lambda\in\CC^*$, the change of variable $t=\sqrt \lambda t$
identifies the two surfaces defined by the equations
$t^2-P(x,y,z)=0$ and 	$t^2-\lambda P(x,y,z)=0$.

Let us illustrate integration with respect to Euler characteristic by
computing $\chi(S)$. The map $\pi$ is two-to-one over
$\CP^2\setminus C$ and one-to-one over $C$,  so Proposition
\ref{prop:int} gives
	\begin{align*}
	  \chi(\CC S) = & 2\chi(\CC P^2\setminus C)+\chi(C) 
          \\ =& 2\chi(\CC P^2)-\chi(C)
	\\= & 6 +d(d-3).
	\end{align*}

 The surface $S$ comes endowed with the line bundle
 $\pi^*\O_{\CP^2}(1)$, and let $V\subset |\pi^*\O_{\CP^2}(1)|$ be the 2-dimensional
 linear system
corresponding to  the lifts of lines in $\CC P^2$ under the
covering map
 $\pi$.  A general curve $D=\pi^{-1}(L)$ in $V$ is transverse to the ramification
locus $C$, and so is a double cover of $L=\CP^1$ ramified at $d$
points. In particular $D$ is non-singular and
\[
\chi(D)=2\chi(\CP^1)-d=4-d.
\]
The curve $D=\pi^{-1}(L)\in V$ is  singular when $L$ is not transverse
to $C$. If the curve $C$ is generic, then a line $L$ in $\CP^2$
that is not transverse to $C$ is of one of the three following types:
	\begin{itemize}[label=$\circ$]
	\item $L$ has an ordinary tangency at a single point $p$ in
          $C$, and $D$ has a node at $p$ as its only singularity;

	\item  $L$ is tangent to $C$ at an inflection point $p$ of $C$,
          and $D$ has a cusp at $p$ as its only singularity;

	\item $L$ is a  bitangent to $C$ at the points $p$ and $q$, and $D$
          has
          two nodes at $p$ and $q$ as its only singularities.
	\end{itemize}
This description of curves in $V$ provides a stratification of $V$
into the following locally closed subsets:
\[
\begin{array}{ll}
  U_0=\{D\in V \mbox{ is non-singular} \}\qquad &
U_1=\{D\in V \mbox{ has a single node} \}
\\  U_2=\{D\in V \mbox{ has a cusp} \}\qquad \qquad &
U_{11}=\{D\in V \mbox{ has two nodes} \}.
\end{array}
\]
By construction, the dual curve $C^\vee$ of $C$ is isomorphic to
$U_1\cup U_{2}\cup U_{11}$. In particular $U_1$ consists of
non-singular points of $C^\vee$, while $U_2$ and $U_{11}$ are cusps
and nodes of $C^\vee$, respectively.

\bigskip
Suppose now that the curve $C$ is defined by a real polynomial
$P\in\RR[x,y,z]$. As the degree is even, the sets $\{\pm P>0\}$ are well-defined. Thus, the real curve $\RR C$ splits $\RR P^2$ into two components $\{\pm P>0\}$, one of them being non-orientable. Let $\RR P^2_-$ be the non-orientable component and $\RR P^2_+$ its complement. We assume that $P$ is chosen such that $\RR P^2_+=\{P>0\}$. The surface $S$ can be equipped with the two real
structures given by the two equations
\[
        	t^2\mp P(x,y,z)=0
\]
in $\CP^3(1,1,1,k)$,
and we denote by $S_{\pm}$ the corresponding real algebraic surfaces.
By construction $\RR S_{\pm}$ is the double cover of $\RP^2_\pm$
ramified along $\RR C$. In particular 
\[
\chi(\RR S_\pm)=2\chi(\RP^2_\pm)+\chi(\RR C)=2\chi(\RP^2_\pm).
\]
Note that  $\chi(\RR P^2_+)+\chi(\RR P^2_-)=\chi(\RR P^2)=1$. 

The linear system $V$ is also real, and the above stratification of
$V$ induces a stratification of $\RR V$. The stratum
$\RR U_{11}$ may be refined, since a real node of 
a real algebraic curve
may be of two types depending on its local real equation:
$y^2-x^2=0$ (hyperbolic node) or $x^2+y^2=0$ (elliptic node).
 The two local branches of the curve 
 are both real at a hyperbolic node, and
are complex conjugated at an elliptic one.
We refine accordingly the stratum $\RR U_{11}$ as
\[
\RR U_{11}^+=\{\mbox{elliptic nodes of }\RR C^\vee \}
\qquad\mbox{and}\qquad
\RR U_{11}^-=\{\mbox{hyperbolic nodes of }\RR C^\vee \}.
\]
The strata $\RR U_{11}^+$ and $\RR U_{11}^-$ 
correspond to real bitangents of $C$ of
type $t_0$ and $t_2$, respectively.
The stratification of $\RR V$ is depicted
in Figure \ref{figure-stratification-discriminant}. 
	\begin{figure}
	\begin{center}
          \includegraphics[scale=1.5]{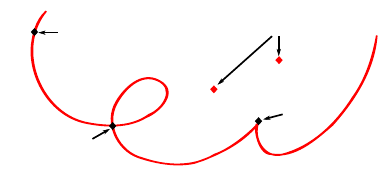}
	  \put(-230,110){curves with}\put(-230,100){a single node}
	  \put(-270,20){curves with}\put(-277,10){two real nodes}
	  \put(-145,115){curves with two complex}\put(-129,105){conjugated nodes}
   	  \put(-80,60){curves with}\put(-80,50){a cusp}
	\caption{\label{figure-stratification-discriminant}Stratification
       of $\RR V$.}
	\end{center}
	\end{figure}

\subsection{Hilbert squares and their Euler characteristic}\label{sec:hilbsch}
The \emph{$n$-pointed Hilbert scheme} $X^{[n]}$ of an algebraic variety
$X$ parameterizes  closed sub-schemes of $X$
of dimension $0$ and length $n$.
It comes equipped with  the \emph{Hilbert-Chow morphism}
$\tau:X^{[n]}\to X^{(n)}$, where
$X^{(n)}=X^n/\mathfrak S_n$ is the $n^{th}$-symmetric power of $X$, that
maps a length $n$ sub-scheme to its support.
The Hilbert scheme $X^{[n]}$
 is a compactification
of the configuration space
of $n$ distinct points on $X$ when $X$ is proper. For example, one has $X^{[1]}=X$, and
the Hilbert
scheme $C^{[n]}$ of a non singular algebraic curve $C$ is nothing but
 $C^{(n)}$. In particular it is
non-singular. All Hilbert schemes  $S^{[n]}$ of a non-singular surface
$S$
are also non-singular by \cite{Fog68}.
More general Hilbert schemes are quite complicated, luckily
the particular ones  needed in this paper
have a simple description.

The  Hilbert square $X^{[2]}$
of a non-singular algebraic variety $X$ is
obtained as follows: the involution $\sigma$ on $X\times X$ permuting the
two factors extends to 
its blow-up 
along the diagonal $\Delta$, and
the quotient is isomorphic to $X^{[2]}$. In particular, points
in $\tau^{-1}(X^{(2)}\setminus\Delta)$ 
correspond to unordered pairs of distinct points on $X$,
while  points in $\tau^{-1}(\Delta)$  correspond to
the choice of a point $p$ on $X$ 
along
with a tangent direction in $\mathbb P(T_pX)$.

\begin{example}
  If $S$ is a non-singular algebraic surface, then one has
  \begin{align*}
  \chi(S^{[2]})=&\frac{\chi([S\times S]\setminus\Delta)}2
  +\chi(\CP^1)\chi(S).
  \\ =& \frac{\chi(S)^2+3\chi(S)}2.
  \end{align*}
\end{example}

It follows from \cite{ran2005note} and \cite{lee2012note} that the
same description of $C^{[2]}$ holds true for a  singular
algebraic curve
with
nodes and cusps as its only singularities. Let us assume 
 that $C$ is contained in a non-singular algebraic surface $S$ to
 lighten up he exposition.
Recall that  nodes and cusps of a complex algebraic curve are points
for which the curve admits the local equation 
\[
y^2+x^2=0 \qquad\mbox{and}\qquad
y^2+x^3=0
\]
in $\CC^2$, respectively. In particular $C$ has a 2-dimensional
tangent space at such singular point, and we \st{see} have that
 $\tau^{-1}(\Delta)\subset C^{[2]}$
 is homeomorphic to
$\mathrm{bl}^{-1}(C)$, where $\mathrm{bl}:\widetilde S\to S$ is the blow-up of $S$ at
all singular points of $C$.
Thanks to this description, we have the following lemma.
\begin{lem}\label{lem:cx ec nc curve}
Given $C$  a compact complex algebraic curve of arithmetic genus $g$,
with $k$ nodes as its only singularities, one
 has
\[
\chi(C^{[2]})={{2g-2}\choose{2}} +k(2-2g) + \frac{k(k+3)}{2}.
\]
Given $C$  a compact complex algebraic curve of arithmetic genus $g$,
with a  cusp as its only singularity, one
 has
\[
\chi(C^{[2]})={{2g-2}\choose{2}} +4(2-g).
\]
\end{lem}
\begin{proof}
  In the first case, one has  $\chi(C)=2-2g+k$. 
Integrating with respects to the Euler characteristic, one obtains
  \begin{align*}
    \chi(C^{[2]})&= \frac{\chi(C)^2-\chi(C)}{2} + \chi(C)+k,
  \end{align*}
  which gives the result after simplification. The case of a cuspidal
  curve is analogous starting from $\chi(C)=2-2g+2$. 
\end{proof}

The above discussion has an analogous real version.
When  $X$ is real, the Hilbert scheme $X^{[n]}$ has a tautological
real structure
inherited from the product real structure on $X^n$.
In particular, real points of $X^{[2]}$ are of three types:
\begin{enumerate}
\item unordered pairs of points in $\RR X$;
\item unordered pairs of complex conjugated points in
  $X\setminus \RR X$;
\item
  points in $\RR \tau^{-1}(\Delta)$.
\end{enumerate}
If $X$ is non-singular, then
$\RR \tau^{-1}(\Delta)=\mathbb P(T(\RR X))=\RR\mathbb P(TX)$, which corresponds to the choice of a real point with a real tangent direction.
\begin{example}
  If $S$ is a non-singular real algebraic projective surface, integrating with respect to the Euler characteristic using the above stratification, one
  has
  \[
  \chi(\RR S^{[2]})=\frac{\chi(S) +\chi(\RR S)(\chi(\RR S)-2)}{2}.
  \]
\end{example}

Suppose now that $C$ is a real algebraic curve with 
nodes and cusps as its only singularities, which
is contained in a non-singular real algebraic surface $S$.
the blow-up $\widetilde S$ of $S$ at
all singular points of $C$ is still a real surface, and
 $\RR\tau^{-1}(\Delta)\subset C^{[2]}$
 is homeomorphic to
 $\RR \mathrm{bl}^{-1}(C)$.
 Next lemma is  a straightforward adaptation of Lemma \ref{lem:cx ec nc curve}.

\begin{lem}\label{lem:r ec nc curve}
Let $C$ be a compact real algebraic nodal curve of arithmetic genus $g$,
with $k_+$ hyperbolic nodes, $k_-$
elliptic nodes, and $k_\CC$ pairs of complex conjugated nodes.
Then one has
\[
\chi(\RR C^{[2]})=1-g +k_\CC -k_+ + {{k_--k_+}\choose{2}}.
\]
Given $C$  a compact real algebraic curve of arithmetic genus $g$,
with a  cusp as its only singularity, one
 has
\[
\chi(\RR C^{[2]})=1-g.
\]
\end{lem}
\begin{example}\label{exa:ec hs}
  The values of $\chi(\RR C^{[2]})$ listed in Table \ref{tab:ec hs}
  will be of particular interest.
  \begin{table}[!h]
\begin{center}
\begin{tabular}{ c||c| c|c|c|c|c|c}
  $(k_+,k_-,k_\CC)$  & $(0,0,0)$ & $(1,0,0)$
  & $(0,1,0)$& $(2,0,0)$ & $(0,2,0)$  & $(1,1,0)$ & $(0,0,1)$ 
\\\hline
$\chi(\RR C^{[2]})-(1-g)$& $0$ & $0$ & $0$ & $1$& $1$ &$-1$  & $1$ 
\end{tabular}

\end{center}
\caption{Some values of  $\chi(\RR C^{[2]})$ for nodal curves}\label{tab:ec hs}
\end{table}
\end{example}

\subsection{Proof of Theorem \ref{thm:signed}}
\label{Sect:ProofFirstMainTheorem}
We use  notations of Section \ref{sec:double}. Let
 $C$ be a generic real algebraic curve of
 degree $d=2k$ in $\CC P^2$, and  let $\pi:S\to\CP^2$ be the  double
 covering of $\CP^2$ ramified along $C$  equipped with any of the two
   lifts of the complex conjugation on $\CP^2$.
Let $n_+$ and $n_-$ be the number of binodal curves in $\RR V$ with
an even and odd number of elliptic nodes, respectively.
Notice that the change of real structure over $S$ switches elliptic
and hyperbolic nodes, so that these two numbers only depend on $C$. 
The left hand-side in the next theorem may be interpreted as the number of
binodal real curves in $V$, counted with Welschinger signs
(see \cite{Wel1}).
	\begin{theo}\label{thm:signed2}
          One has
	  \[
          n_+ - n_- =-2\chi(\RP^2_+)\chi(\RP^2_-) + \frac{ d(d-2)}{2}.
          \]
	\end{theo}
        \begin{proof}
 The proof goes by computing in two ways the Euler characteristic of
 $\RR\CCC^{[2]}$, using  Proposition \ref{prop:int r} and the two
 maps $\nu_1$ and $\nu_2$.
 To that end we first have to choose a real structure on $S$, that we
 assume to be $S_+$.
 We start with the map $\nu_1:\RR \CCC^{[2]}\to\RR V$.
 Since all curves in $V$ have the same arithmetic genus $1-\frac{4-d}{2}=k-1$, one has
by Table \ref{tab:ec hs} and Lemma \ref{lem:r ec nc curve},
	\begin{align*}
	  \chi(\RR\CCC^{[2]}) = & (2-k)\chi(\RR V) +n_+-n_-
         \\ =&  (2-k) +n_+-n_-.\\
	\end{align*}

        Next we use the map $\nu_2:\RR \CCC^{[2]}\to \RR  S_+^{[2]}$.
        Recall that   $\RR S_+^{[2]}$ is stratified along the following
        three types of real points.
	\begin{enumerate}
	\item[(1)] Pairs of distinct real points $p,q\in\RR S_+$.
          \begin{itemize}
            \item   If
          $\pi(p)\ne \pi(q)$, then
          there exists a unique real curve in $V$ passing through
          $\{p,q\}$: the lift of the line passing through $\{\pi(p),\pi(q)\}$.
          \item   If
          $\pi(p)=\pi(q)$, then
          there exists a pencil of real curves in $V$ passing through
          $\{p,q\}$.
          \end{itemize}
          Altogether, since $\chi(\RP^1)=0$,
          this stratum of $\RR S_+^{[2]}$  contributes for
	\[
        \frac{\chi(\RR S_+)^2-\chi(\RR S_+)}{2}-\chi(\RR P^2_+)=-2\chi(\RP^2_+)\chi(\RP^2_-)
        \]
        to $\chi(\RR S_+^{[2]})$.
	\item[(2)] Pairs of complex conjugate points
          $(p,\overline{p})$ in $S_+\setminus\RR S_+$.
          In this case
          \[
          \pi(p)=\pi(\overline p)\Longleftrightarrow
          \pi(p)\in\RP^2 \Longleftrightarrow  \pi(p)\in\RP^2_-. 
          \]
          \begin{itemize}
            \item   If
          $\pi(p)\notin \RP^2_-$, 
          there exists a unique real curve in $V$ passing through
          $\{p,\overline p\}$.
          \item   If
          $\pi(p)\in \RP^2_-$, then
          there exists a pencil of real curves in $V$ passing through
          $\{p,\overline p\}$.
          \end{itemize}
Altogether,
          this stratum of $\RR S_+^{[2]}$  contributes for
	\begin{align*}
	  \frac{ \chi(S_+)-\chi(\RR S_+)}{2} -\chi(\RR P^2_-) 
	 =  2+k(2k-3)
	\end{align*}
        to $\chi(\RR S_+^{[2]})$.
        
 \item[(3)] Points $p$ in $\RR S_+$  equipped with
        a real tangent line $l\in \PP(T_p\RR S_+)$. 
		\begin{itemize}
		\item If $p\notin \RR C$, then there exists
                  a unique real curve $D$ in $V$ passing through $p$ and
                  such that $T_pD=l$:
                  the lift
                  of the line passing through $\pi(p)$ with direction
                  $\mathrm{d}_p\pi(l)$.
                  Since
                  $\mathbb P\left(T\left(\RR S_+\setminus \RR C\right)\right)$
                  is a locally trivial fibration over
                  $\RR S_+\setminus \RR C$ with fiber $\RP^1$,
                  this stratum of $\RR S_+^{[2]}$ does not contribute 
                    to $\chi(\RR S_+^{[2]})$.

		  \item If $p\in \RR C$ and $\mathrm{d}_p\pi(l)=0$, then
                    $l\subset T_pD$ for
                     any curve $D$ passing through
                    $p$. Since there is a
                    pencil of such curves,   this stratum of $\RR
                    S_+^{[2]}$ does not contribute  
                    to $\chi(\RR S_+^{[2]})$.

\item If $p\in \RR C$ and $\mathrm{d}_p\pi(l)\ne 0$,
  then $l\subset T_pD$ 
  only for the lift $D$ of the tangent line of $C$ at $p$. 
  Hence this stratum of $\RR  S_+^{[2]}$  contributes for
  \[
  \chi_c(\RR) \chi(\RR C)=0
  \]
                    to $\chi(\RR S_+^{[2]})$, since $\RR C$ is a
                    disjoint union
                    of circles.
		\end{itemize}
	\end{enumerate}
	In the end, we obtain
	\begin{align*}
	\chi(\RR\CCC^{[2]}) = &-2\chi(\RP^2_+)\chi(\RP^2_-)+ 2+k(2k-3).
	\end{align*}	
        Equating these two computations of $\chi(\RR\CCC^{[2]})$ gives
        \[
        (2-k) +n_+-n_-=-2\chi(\RP^2_+)\chi(\RP^2_-)+ 2+k(2k-3),
        \]
        	which is the desired identity.
        \end{proof}

\medskip
\begin{proof}[Proof of Theorem \ref{thm:signed}]
  To deduce Theorem \ref{thm:signed} from Theorem \ref{thm:signed2},
  it remains to notice the two following trivial
  facts, whatever the chosen real structure on
  $S$ is:
  \begin{enumerate}
  \item[(1)]   a real binodal curve in $\RR V$ has an odd number of elliptic nodes
  if and only if it is the lift of a real bitangent $L$ of $C$ with
  $\epsilon(L)=-1$;
\item[(2)] one has
  \[
  \chi(\RP^2_+)= p-n
  \qquad\mbox{and}\qquad
    \chi(\RP^2_-)=n-p+1.
  \]
  \end{enumerate}
\end{proof}

\subsection{Plücker and Klein Formulas for plane curves of even degree}
Since it may be of interest,
we briefly indicate how our strategy  to prove Theorem \ref{thm:signed} also allows to 
 recover Plücker and Klein Formulas
for plane curves of even degree.

\medskip
\subsubsection{Number of inflection points.}
Using the map $\nu_1:\CCC\to V$, one obtains
\[
\chi(\CCC)=(4-d)\chi(\CP^2) + \chi(U_1) + 2t_\CC(d)+2i_\CC(d).
\]
Recall that  $U_1\cup U_2\cup U_{11}$ is the dual curve $C^\vee$ of
$C$. Since $C$ is the normalization of $C^\vee$, and  $U_{11}$ and
$U_2$ are the nodes and the cusps of $C^\vee$, respectively, one has
\[
\chi(U_1) + 2t_\CC(d)+i_\CC(d)=\chi(C)=-d(d-3).
\]
Since a pencil of lines passes through any point in $\CP^2$,
considering the map $\nu_2:\CCC\to S$ gives
\[
\chi(\CCC)=\chi(\CP^1)\chi(S)=12+2d(d-3).
\]
Combining these three identities, one obtains
\[
i_\CC(d)=3(d-4) +d(d-3)+2d(d-3)+12 =3d^2-6d=3d(d-2).
\]

\medskip
\subsubsection{Number of bitangents.}
Using the map $\nu_1:\CCC^{[2]}\to V$ and Lemma \ref{lem:cx ec nc curve}, one obtains
\begin{align*}
  \chi(\CCC^{[2]})=&
  \frac{(d-4)(d-5)}{2}\chi(\CP^2) + (6-d)\chi(U_1) +
  (13-2d)t_\CC(d)+(12-2d)i_\CC(d)
  \\ =& 3\frac{(d-4)(d-5)}{2} +  d(d-3)(d-6) + 3d(d-2)(6-d)
   + t_\CC(d)
  \\ =& 30- \frac12 d(d-3)(4d-21)
   + t_\CC(d).
\end{align*}
Next we use the map
$\nu_2:\CCC^{[2]}\to S^{[2]}$. Analogously to the proof of Theorem
        \ref{thm:signed}, we use the following stratification of
         $S^{[2]}$ into
        two types of subschemes.
	\begin{enumerate}
	\item[(1)] Pairs of distinct real points $p,q\in S$.
          Then  there exists a unique curve in $V$ passing through
          $\{p,q\}$ if       $\pi(p)\ne \pi(q)$, and
          a pencil of such curves otherwise.
          Hence
          this stratum of $S^{[2]}$  contributes for
	\[
        \frac{\chi(S)^2-\chi(S)}{2}+\chi(\CP^2)-\chi(C)=18+\frac12 d(d-3)(d(d-3)+13)
        \]
        to $\chi(S^{[2]})$.
	\item[(2)]  Point $p$ in $S$  equipped with
        a  tangent line $l\in \PP(T_pS)$. 
		\begin{itemize}
		\item If $p\notin C$, then there exists
                  a unique curve $D$ in $V$ passing through $p$ and
                  such that $T_pD=l$.
                  Since
                  $\mathbb P\left(T\left(S\setminus C\right)\right)$
                  is a locally trivial fibration over
                  $S\setminus C$ with fiber $\CP^1$,
                  this stratum of $S^{[2]}$  contributes 
                  for
	\[
        \chi(\CP^1)(\chi(S)-\chi(C))=12+4d(d-3)
        \]
        to $\chi(S^{[2]})$.

		  \item If $p\in C$ and $\mathrm{d}_p\pi(l)=0$, then
                    $l\subset T_pD$ for
                     any curve $D$ in the pencil passing through
                     $p$. This stratum of
                     $S^{[2]}$ then contributes   
                             for
	\[
        \chi(\CP^1)\chi(C)=-2d(d-3)
        \]
         to $\chi(S^{[2]})$.

\item If $p\in C$ and $\mathrm{d}_p\pi(l)\ne 0$,
  then $l\subset T_pD$ 
  only for the lift $D$ of the tangent line of $C$ at $p$. 
  Hence this stratum of $S^{[2]}$  contributes for
  \[
  \chi_c(\CC) \chi(C)=-d(d-3)
  \]
                    to $\chi(S^{[2]})$.
		\end{itemize}
	\end{enumerate}
	In the end, we obtain
	\begin{align*}
	\chi(\CCC^{[2]}) = &30 +\frac12 d(d-3)(d(d-3)+15).
	\end{align*}	
        Equating these two computations of $\chi(\CCC^{[2]})$, we
        obtain
        \[
 t_\CC(d)=\frac 12 d(d-2)(d-3)(d+3).
 \]

        \medskip
        \subsubsection{Proof of Klein's formula.}
       
        The following is a reformulation in our setting
        of the proof by Viro \cite{Viro-Euler}  of
        Klein Formula in any degree.
 As $\RR\CCC$ is of dimension
 $3$, its Euler characteristic is $0$.  
We consider instead
$A=\nu_1^{-1}(\RR V)\subset \CCC$.
 Note first that
 \[
 0=\chi(\RR C)=\chi(\RR U_1)+ 2t_2 +i.
 \]
Using the map $\nu_1:A\to \RR V$, one obtains
 	\begin{align*}
	\chi(A) = & (4-d)\chi(\RR V) + \chi(\RR U_1)+2t_2+2i+2t_0 \\
	= & (4-d)+i+2t_0.
	\end{align*}

Next we compute $\chi(A)$ using the map $\nu_2:A\to S$.
	\begin{itemize}
	\item If $\pi(p)\notin\RR P^2$, the only real curve in $V$
          passing through $p$ is the lift of the real line passing through $\pi(p)$ and $\overline{\pi(p)}$.
	\item If $\pi(p)\in\RR P^2$, then the lift of any real line
          from the pencil 
          passing through $\pi(p)$ yields a real curve in $V$ passing through $p$.
	\end{itemize}
        Therefore, since $\chi(\RP^1)=0$, we get that
	\begin{align*}
	\chi(A) =  \chi(S-\pi^{-1}(\RR P^2)) = \chi(S)-2\chi(\RP^2)=4+d(d-3).
	\end{align*}
        Equating these two computations of $\chi(A)$ gives
        \[
        i+2t_0=d(d-2).
        \]

\section{Bitangents of a perturbation of a multiple conic}\label{sec:conic}

In this section we explain how to locate  real bitangents of a generic perturbation of a multiple real conic in $\RP^2$.
Our study  refines the control of the topology of the real part of such perturbation,  well known to experts in Hilbert's 16th problem.
Such perturbations can be achieved by a standard application of the deformation to normal cone technique, and they admit a very concrete and constructive formulation.
Unfortunately we could not find such constructive presentation in the literature, which we believe could be of interest outside the scope of Hilbert's 16th problem.
Hence we start by recalling in Section \ref{sec:oldconstr} how to perturb a real multiple conic while controlling its topology, before focusing on  real bitangents of such perturbations in Section \ref{sec:real bit}.

\subsection{Controlling the topology of a perturbation of a multiple real conic}\label{sec:oldconstr}
We first state this construction
in its  abstract version in Proposition \ref{prop:conic1},
before providing a
 constructive version  in Proposition \ref{prop:conic2}.
 
A non-singular compact algebraic curve $C$ in $\mathcal O_{\CC P^1}(n)$ is said to be a \emph{$k$-section} if
the projection
$\pi: \mathcal O_{\CC P^1}(n)\to \CP^1$ restricts to a  degree $k$
ramified covering $C\to \CP^1$.
Throughout the text, the real structure on both $\CP^1$ and $\mathcal O_{\CC P^1}(n)$
is induced by the standard complex conjugation on $\CC$. 
\begin{prop}\label{prop:conic1}
  Let $\C_0$ be a non-singular real conic in $\CP^2$ with
  $\RR\C_0\ne\emptyset$, and let $N_0$ be
  an open
  tubular neighborhood of $\RR \C_0$  in $\RP^2$.
  Let $C$ be a  real $k$-section in
  $\mathcal O_{\CC P^1}(4)$, and let 
  $\phi:\RR \mathcal O_{\CC P^1}(4)\to N_0$ be any homeomorphism.
  Then  there exists a non-singular real algebraic curve $\C$ of degree $2k$
  in $\CP^2$ such that $\RR \C$ is isotopic to $\phi(\RR C)$ in $\RP^2$. 
\end{prop}
\begin{remark}
Since $\RP^2\setminus \RR \C_0$ is the disjoint union of a disk an a
Möbius band, there are at most two isotopy classes  for the image
$\phi(\RR C)$ in $\RP^2$. Given $\phi(\RR C)$ one such image, the
potential second isotopy class is obtained  by composing $\phi$ on the left by
the vector bundle automorphism
$s:\mathcal O_{\CC P^1}(4)\to \mathcal O_{\CC P^1}(4)$ 
that restricts to $-\mathrm{Id}_F$  on each fiber $F$.
\end{remark}
\begin{proof}
  As already mentioned, this is a standard application of the deformation to normal cone technique.
 Let $\X$ be the blow-up of $\CP^2\times\CC$ along
 $\C_0\times\{0\}$. The projection
 $\CP^2\times\CC\to \CC$ extends to a projection
 $\pi:\X\to\CC$, and we denote by $X_t$ the fiber over $t\in\CC$. By
 construction $X_t=\CP^2$ for $t\ne 0$. Since $\C_0$ is rational and has
 self-intersection 4 in $\CP^2$, the surface $X_0$ is the union of
 $X_{0,1}=\CP^2$ and $X_{0,2}=\PP(\mathcal O_{\CC P^1}(4) \oplus \mathcal O_{\CC P^1})$ with
 $X_{0,1}\cap X_{0,2}$ being  $\C_0$ in $X_{0,1}$ and the $(-4)$-section $E_4$
 in  $X_{0,2}$.
 We identify $\C_0$ with the $0$-section of  $\mathcal O_{\CC
   P^1}(4)$.

 Now let $\L$ be a line bundle on $\X$ such that
 \begin{itemize}[label=$\ast$]
 \item $\L_{|X_t}=\O_{\CP^2}(2k)$ if $t\ne 0$;
 \item $\L_{|X_{0,1}}=\O_{\CP^2}$ ;
 \item $\L_{|X_{0,2}}=\O_{X_{0,2}}(k\C_0)$.
 \end{itemize}
 and
 let $\sigma_0$ be a real section of $H^0(X_{0,2}; \O_{X_{0,2}}(k\C_0))$
 defining $C$.
 Since $E_4\cap C=\emptyset$ by assumption,
the section
 $\sigma_0$ extends to a real section
 $\sigma\in H^0(\X;\L)$. Denote by $\Y$ the  divisor of $\X$ defined by
 $\sigma$. Note that $\Y\cap X_0=\Y \cap (X_0\setminus X_{0,1})=C$.
 Since $C$ is non-singular,
 the intersection of $\Y $ and $X_{0,2}$ is transverse. 
 Since  $X_t$ is obtained by smoothing $X_0$, whose singular locus
 $X_{0,1}\cap X_{0,2}$ is disjoint from $C$, the curve
$\Y\cap X_t$, with $t\in\RR^*$ either positive or negative,
is the curve $\C$ whose existence in claimed in the proposition. 
Note that if $t_1$ and $t_2$ are of different signs,
the curves $\Y\cap X_{t_1}$ and $\Y\cap X_{t_2}$ realize the two
(potentially equal)
isotopy classes $\phi(\RR C)$ and $\phi\circ s(\RR C)$ in $\RP^2$.
\end{proof}

Interpreting a situation  as a special instance of a
general and powerful theory may provide elegant and conceptual
proofs,
this may  nevertheless not be  as useful to work out  concrete examples. 
To that purpose let us now express Proposition \ref{prop:conic1} and
its proof in
a
more down-to-earth and
constructive way.

Recall that the Newton polygon $\Delta(P)$ of a
polynomial $P(u,v)\in \CC[u,v]$ is
the convex hull in $\RR^2$ of the set of all  $(i,j)\in\ZZ^2$ such that
the $u^iv^j$-coefficient of $P$ does not vanish.
\begin{defi}\label{def:ksection}
  An algebraic curve $C$ in $\CC^2$ defined by a polynomial
$P(u,v)=\sum_{(i,j)\in\ZZ^2} a_{i,j}u^iv^j\in \CC[u,v]$
  is called a \emph{non-singular $k$-section}
  (of degree 4) if the following conditions are satisfied:
\begin{enumerate}
  \item\label{cond1} $\Delta(P)$ is the polygon $\Delta_k$
    depicted in Figure \ref{fig:deltak};
    \begin{figure}[h!]
\centering
\begin{tabular}{c}
  \includegraphics[width=4cm, angle=0]{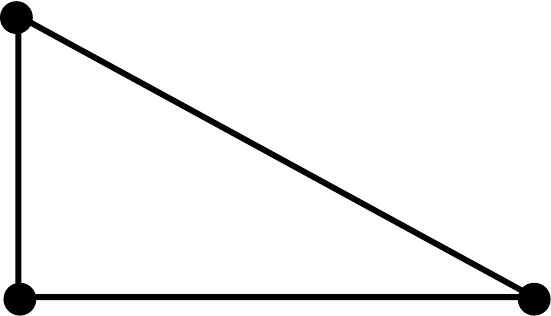}
  \put(-145,0){$(0,0)$}
  \put(5,0){$(4k,0)$}
  \put(-145,60){$(0,k)$}
\end{tabular}
\caption{The polygon $\Delta_k$.}
\label{fig:deltak}
    \end{figure}
    
\item  $C$ is non-singular in $\CC^2$;

\item\label{cond2} the polynomial
  \[
  \sum_{i+4j=4k} a_{i,j}u^iv^j
  \]
  has no multiple factors. 
\end{enumerate}
If $P(u,v)$ can be chosen real, the $k$-section $C$ is said to be real.
\end{defi}

\begin{example}\label{ex:2sec}
  The curve $C_2$ defined by the equation
  \[
  v^2-u(u-1)(u-2)(u-3)(u-4)(u-5)(u-6)(u-7)=0
  \]
  is a real  $2$-section, whose real part is depicted in
  Figure \ref{fig:ex sec}a. 
    \begin{figure}[h!]
\centering
\begin{tabular}{ccccc}
  \includegraphics[width=3cm, angle=0]{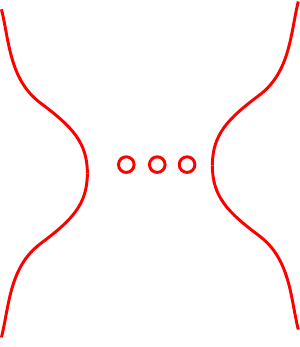}
  &\hspace{5ex}&
  \includegraphics[width=3cm, angle=0]{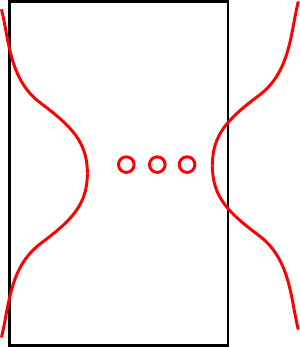}
  &\hspace{5ex}&
  \includegraphics[width=3cm, angle=0]{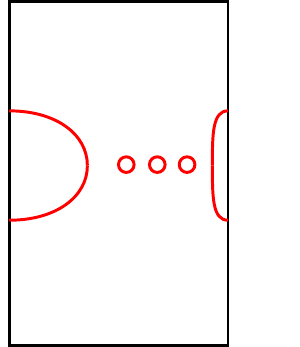}
  \\
  \\ a) $\RR C_2$ && b)  $R_M$ and $\RR C_2$&& c)$R_M$  and  $h(\RR C_2\cap R_M,1)$ 
\end{tabular}
\caption{The $2$-section $C_2$}
\label{fig:ex sec}
    \end{figure}
\end{example}

Let $C\subset \CC^2$ be a real $k$-section.
By condition (\ref{cond1}) in Definition \ref{def:ksection}, the curve
$C$ has $k$ different (complex) asymptotic directions with equation $v=au^4+b$.
Hence for $M$ a  positive (large enough) real number, the square  $R_{M}$ with
vertices $(\pm M,\pm M^5)$ satisfies the following (see Figure \ref{fig:ex sec}b):
\begin{itemize}
\item the horizontal edges of $R_{M}$ do not intersect $\RR C$;
\item the intersection $R_{M'}\cap \RR C$ and $R_{M}\cap \RR C$
  are isotopic if $M'\ge M$. 
\end{itemize}
By condition  (\ref{cond2}) in Definition \ref{def:ksection}, we may
further assume that, denoting by $e_{v-}$ and $e_{v+}$ the left and right
vertical edges of $R_M$ respectively,
\begin{itemize}
\item there exists an isotopy $h:R_{M}\times [0;1]\to R_M$, commuting with the
  projection
  $(u,v)\mapsto u$ and mapping the square $\partial R_M$ to itself,
   such that the
  $y$-coordinates of the points
  $  h\left(e_{v-}\cap \RR C,1 \right)$ equal the $y$-coordinates of
  the points
  $h\left(e_{v+}\cap \RR C,1\right)$  (see Figure \ref{fig:ex sec}c).
\end{itemize}
Following \cite{O1}, the pair $(R_M,h(\RR C\cap R_M,1))$ considered up
to isotopy as above 
is called the \emph{$\L$-scheme} of $C$.
Identifying the points on the two vertical edges of $R_M$ having the
same $y$-coordinate, we obtain an annulus $N$. Furthermore
the image $\gamma_C$ of $h(\RR C\cap R_M,1)$ under this identification
is a disjoint union of embedded circles in $N$. 

Up to isotopy, one can embed $N$ in exactly two ways in $\RP^2$, that we
denote by $N_+$ and $N_-$: the bottom horizontal edge of $R_M$ is
mapped inside the top horizontal edge for $N_+$, while   the top
horizontal edge of $R_M$ is
mapped inside the bottom horizontal edge for $N_-$.
We also denote by $\gamma_{C,-}\subset N_-$ and $\gamma_{C,+}\subset N_+$
the image of $\gamma_C$ under these two embeddings.
Writing $\sigma(x,y)=(x,-y)$, the embeddings $N_+$ and $N_-$ are
informally obtained as follows (see Figure \ref{fig:N}):
\begin{itemize}
\item $N_+$: identify the vertical edges of $R_M$ by bending $R_M$
  down;
  \item $N_-$: identify the vertical edges of $\sigma(R_M)$ by bending $\sigma(R_M)$
  down.
\end{itemize}
    \begin{figure}[h!]
\centering
\begin{tabular}{c}
  \includegraphics[width=10cm, angle=0]{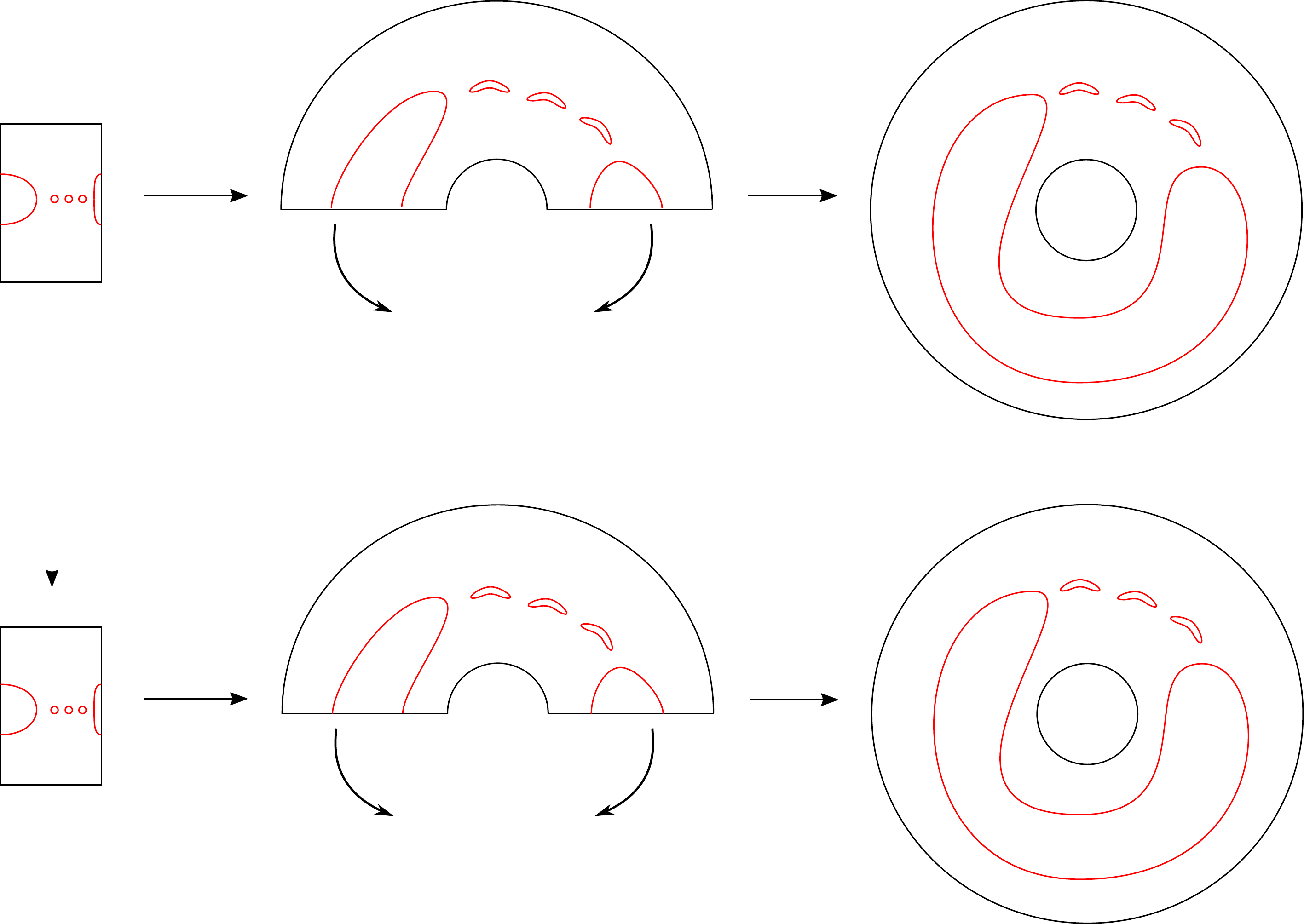}
  \put(-315,155){$R_M$}
   \put(-320,45){$\sigma(R_M)$}
   \put(-267,100){$\sigma$}
 \put(8,155){$N_+$}
  \put(8,45){$N_-$}
\end{tabular}
\caption{$N_-$ and $N_+$}
\label{fig:N}
    \end{figure}

    \begin{example}\label{ex:3sec}
      If $C_2$ is the $2$-section from Example
      \ref{ex:2sec}, the two $\L$-schemes of $C_2$ and $\sigma(C_2)$
      are isotopic, and 
      the two embeddings $\gamma_{C_2,-}$ and   $\gamma_{C_2,+}$ of
      $\gamma_{C_2}$ are isotopic in $\RP^2$, see Figure \ref{fig:N}.
      There exists a real $3$-section $C_3$  with $\L$-scheme 
       as depicted in Figure \ref{fig:3sec}a
      (we stretched $R_M$ in the horizontal direction to make the
      picture more readable).
      The two embeddings 
       $\gamma_{C_3,-}$ and   $\gamma_{C_3,+}$, depicted in Figures
      \ref{fig:3sec}b and \ref{fig:3sec}c respectively,   are not isotopic in $\RP^2$.
      Real sextics with real part isotopic to  $\gamma_{C_3,-}$
      are
     called \emph{Harnack sextics}, and
     real sextics with real part isotopic to   $\gamma_{C_3,+}$ are
     called \emph{Hilbert sextics}.
     \begin{figure}[h!]
\centering
\begin{tabular}{ccccc}
  \includegraphics[width=4cm, angle=0]{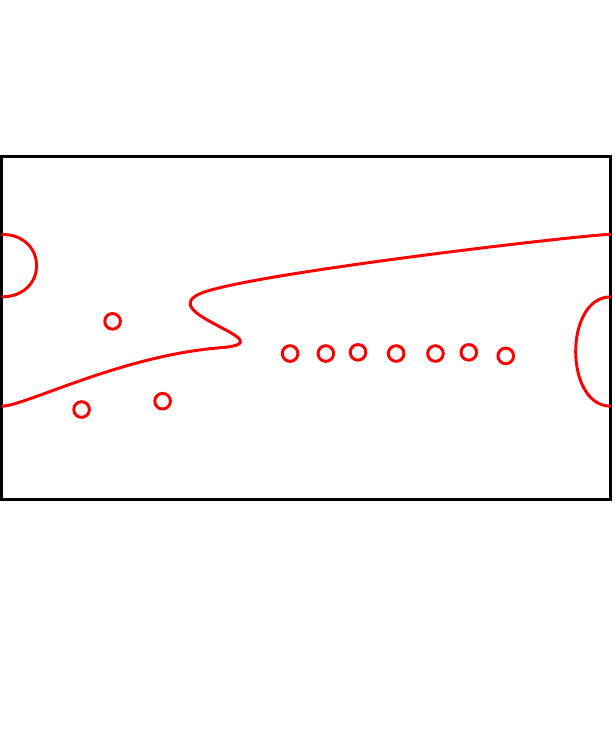}
  &\hspace{5ex}&
  \includegraphics[width=5cm, angle=0]{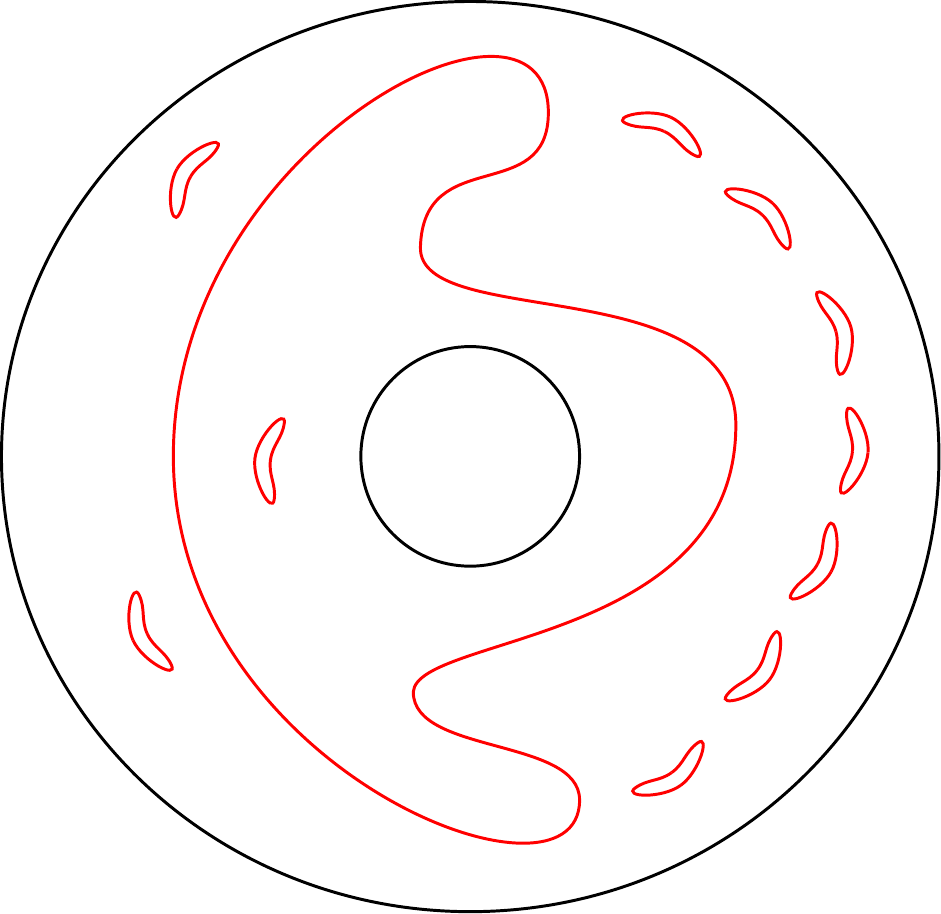}
  &\hspace{5ex}&
  \includegraphics[width=5cm, angle=0]{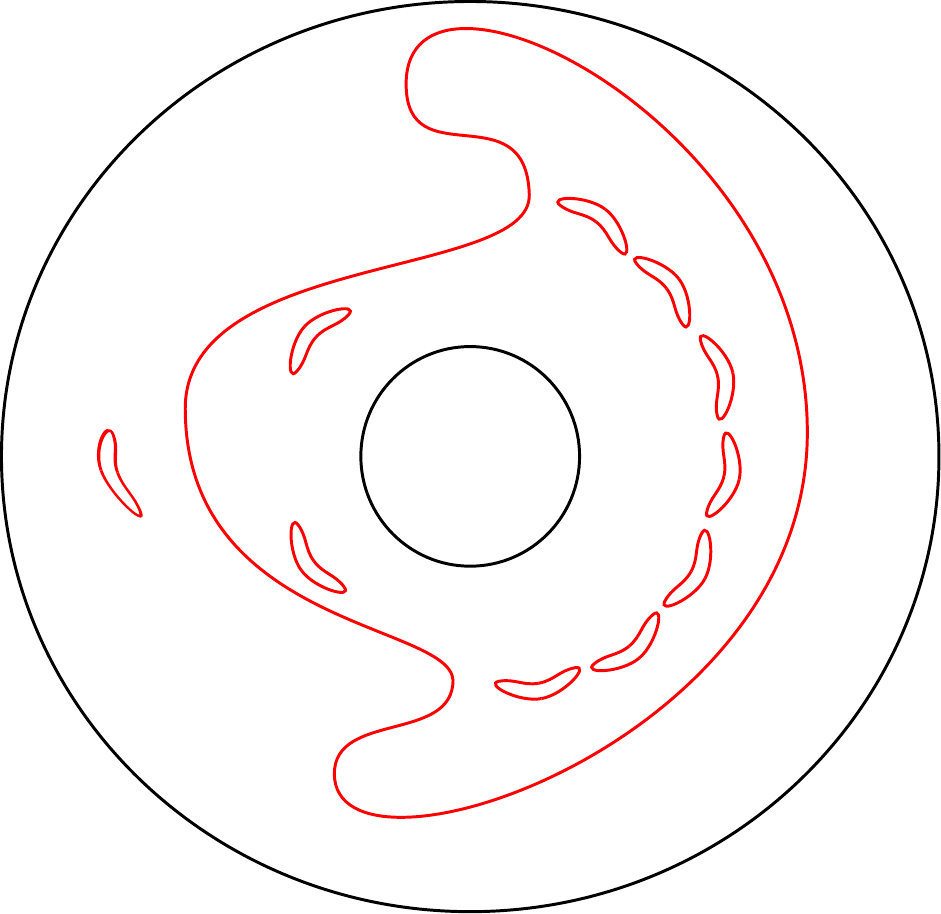}
  \\
  \\ a) $\L$-scheme of $C_3$&& b) $\gamma_{C_3,-}$  && c) $\gamma_{C_3,+}$
\end{tabular}
\caption{The $3$-section $C_3$}
\label{fig:3sec}
    \end{figure}
    
     The curve $C_3$ is a so-called \emph{simple Harnack curve} with
     Newton polygon $\Delta_3$.
     Simple Harnack curves are objects from real algebraic geometry
     that are extremal in
many aspects, see \cite{Mik11,MikRul01} for example.
     There are many ways to construct them
     explicitly, let us mention the three following ones in the case of the 
     $3$-section $C_3$:
     \begin{itemize}
     \item the Harnack method from \cite{Har} using small perturbations; 
     \item combinatorial patchworking using Harnack sign distribution
       from \cite{IV2};
     \item perturbing a rational simple Harnack curve:
     by  \cite{Bru14b},  the $\L$-scheme of the rational curve $C_{3,rat}$ with parameterization
       \[
       \begin{array}{ccc}
         \CC\setminus\{15,16\} &\longrightarrow & \CC^2
         \\ \theta &\longmapsto &
         \left(\frac{\theta(\theta-1)(\theta-2)}{(\theta-15)(\theta-16)},
         \frac{-\prod_{i=3}^{14}(\theta-i)}{(\theta-15)^4(\theta-16)^4}\right)
       \end{array}
       \]
       gives the one depicted in Figure \ref{fig:3sec}a, except that
       each of the 10 ovals is
       contracted to a point. Let $P_{rat}$ be a real
       equation of $C_{3,rat}$. It follows from
\cite[Remark 8]{Br28} that one can
       choose for $C_3$ the curve with equation 
       \[
      P_{rat}(u,v) +\epsilon\left( \frac{\partial P_{rat}}{\partial u}(u,v)
       + \frac{\partial P_{rat}}{\partial v}(u,v) \right)=0,
       \]
       with $\epsilon$ a small enough real parameter.
     \end{itemize}
   \end{example}

\begin{prop}\label{prop:conic2}
  Let $C$ be a real $k$-section in $\CC^2$ defined by the real polynomial
  \[
  P(u,w)=\sum_{i+4j\le 4k}a_{i,j}u^iw^j.
  \]
  Then for  a small enough real parameter $t$, the 
  real algebraic curve $\C_t$ of degree $2k$ in $\RP^2$
  defined by the homogeneous equation
  \[
  \sum_{i_1+i_2+2j= 2k} t^{k-j} a_{2i_1,j} x^{i_1}(y^2-xz)^jz^{i_2}  +
   \sum_{i_1+i_2+2j+1= 2k} t^{k-j} a_{2i+1,j} x^{i_1}y(y^2-xz)^jz^{i_2} =0
  \]
  is non-singular and
  has a real part isotopic in $\RP^2$ to $\gamma_{C,-}$ if $t<0$, and
  to  $\gamma_{C,+}$ if $t>0$.
\end{prop}
\begin{rem}
Plugging $t=0$ in the above equation, we see that  $\C_t$ is a real
perturbation of the $k$-multiple of the conic
$\C_0:y^2-xz=0$.
\end{rem}
\begin{proof}
  Let us start with some general considerations.
  Let $\C_0$ be the real conic in $\CP^2$ with equation
$y^2-xz=0$.
  Fix the coordinate systems $[x:y:z:w]$ on the weighted projective space $V=\CP^3(1,1,1,2)$, 
  and $t$  on the affine line $\CC$ respectively, and consider
  the real hypersurface $\H$ in $V\times\CC$
  with equation
  \[
  tw - y^2-xz=0.
  \]
  By construction, the hypersurface $\H$ comes equipped with a projection
  $\rho:\H\to\CC$, and we define $H_t=\rho^{-1}(t)$.
  For $t\ne 0$,
  the linear projection $V\setminus\{[0:0:0:1]\}\to\{w=0\}=\CP^2$
  restricts to an isomorphism $H_t\to \CP^2$. We use this
  identification of $H_t$ and $\CP^2$ for the rest of the proof.
  The hypersurface $H_0\subset V$ is the cone over $\C_0$ with vertex
  $[0:0:0:1]$, and
  $H_0\setminus\{[0:0:0:1]\}$ is isomorphic to the normal bundle of $\C_0$ in
  $\CP^2$, see Figure \ref{fig:defnc}a for a schematic picture in
  dimension 1. 
  In fact, the deformation to normal cone $\X$ constructed in the
  proof of Proposition \ref{prop:conic1} is the strict transform
  of $\H$ under the blow up of  $V\times\CC$ along  $\{[0:0:0:1]\}\times \CC$.
  In particular the complement $N_0$ in $\RR H_0$ of a small
  neighborhood of $[0:0:0:1]$ deforms in $\RR H_t$ (and so in
  $\RP^2$), with 
  $t\in\RR^*$, to a tubular neighborhood $N_t$ of $\RR\C_0$,
  see Figures \ref{fig:defnc}b and \ref{fig:defnc}c.
       \begin{figure}[h!]
\centering
\begin{tabular}{ccc}
  \includegraphics[width=4.5cm, angle=0]{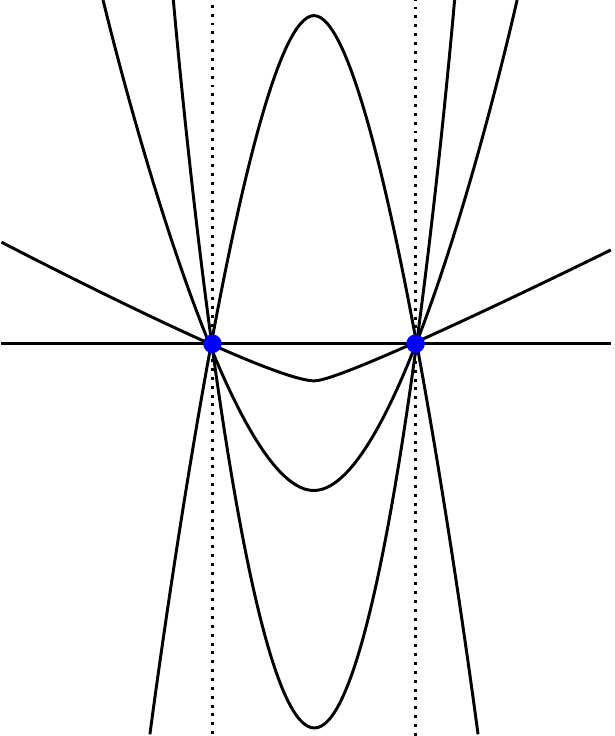}
  \put(-152,100){$\RR H_{t_1}$}
  \put(-152,80){$\RP^2$}
  \put(-126,120){$\RR H_{t_2}$}
  \put(-72,-10){$\RR H_{t_3}$}
   \put(-125,10){$\RR H_{t_4}$}
   \put(-90,160){$\RR H_{0}$}
    \put(-55,160){$\RR H_{0}$}
  \put(-35,70){\textcolor{blue}{$\RR\C_{0}$}}
  \put(-110,70){\textcolor{blue}{$\RR\C_{0}$}}
 &
  \includegraphics[width=4.5cm, angle=0]{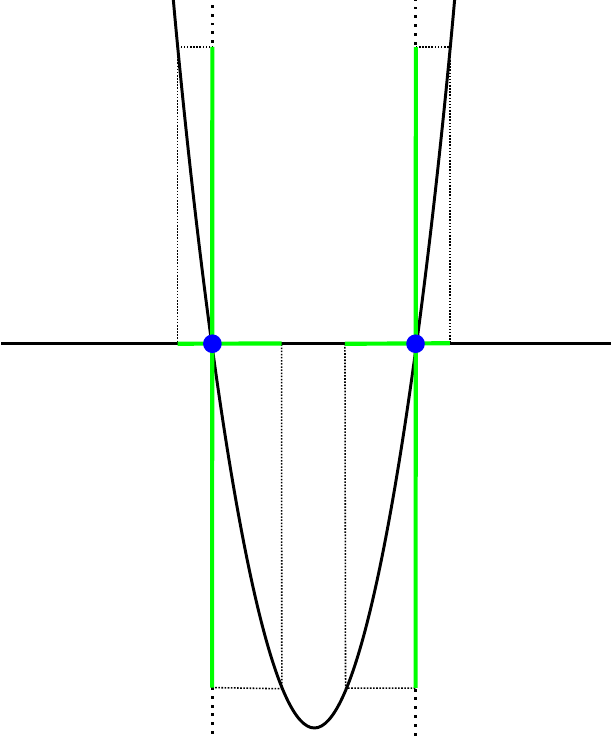}
   \put(-80,130){\textcolor{black}{$N_0$}}
  \put(-75,87){\textcolor{black}{$N_t$}}
 &
  \includegraphics[width=4.5cm, angle=0]{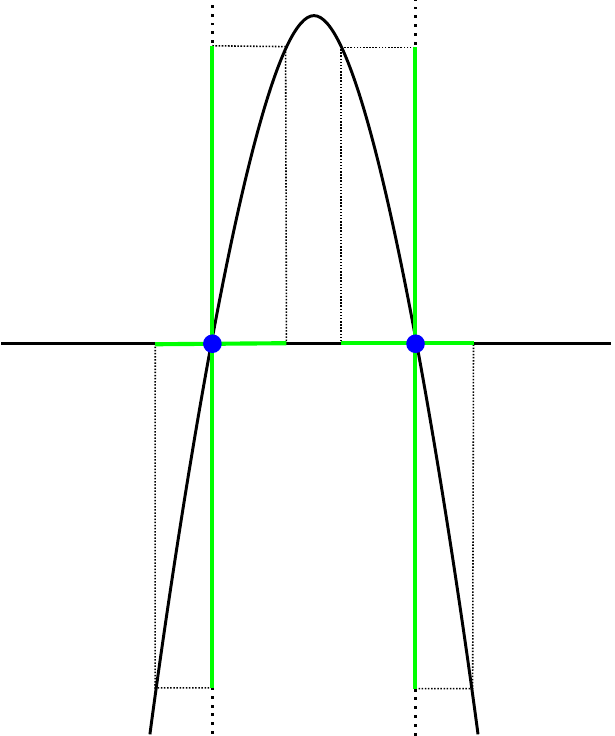}
   \put(-100,130){\textcolor{black}{$N_0$}}
  \put(-77,72){\textcolor{black}{$N_t$}}
  \\
  \\ a) $t_1>t_2>t_3>0>t_4$& b)  & c) 
\end{tabular}
\caption{Deformation to normal cone in image in $\RR V\setminus\{[0:0:0:1]\}$.}
\label{fig:defnc}
    \end{figure}
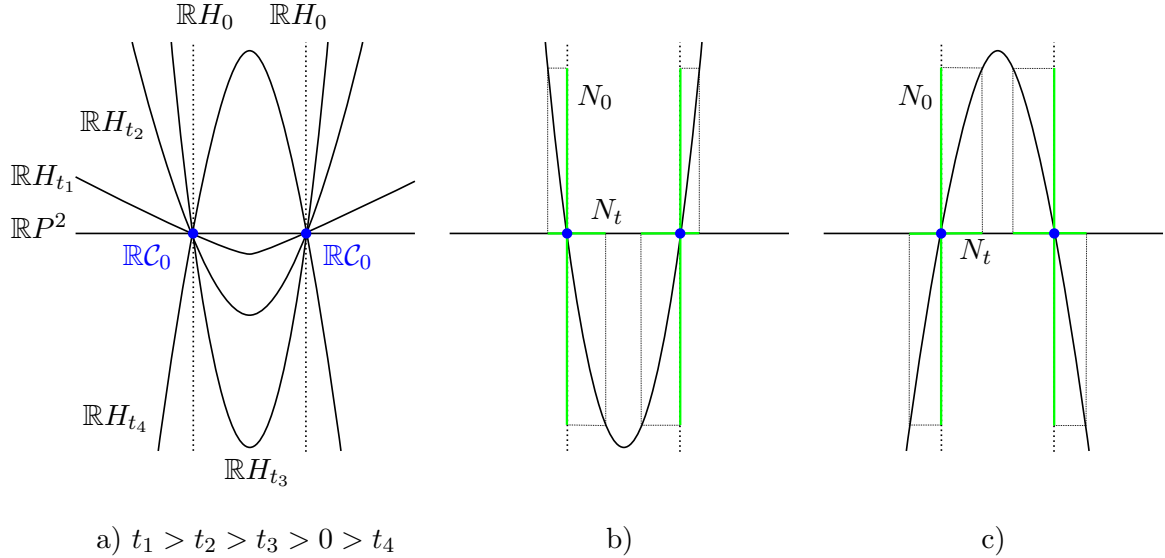

Now suppose that  $\Y$ is a real algebraic hypersurface in $V\times\CC$
intersecting transversally $H_0$ along a $k$-section $\widetilde C$, and define
$\C_t=H_t\cap \Y$ if $t\ne 0$. If follows from the preceding remarks that 
there exists a homeomorphism of pairs
$\phi_t:(N_0,\RR \widetilde C)\to (N_t,\RR \C_t)$
 for
$t$ real and small enough, see  Figure \ref{fig:defnc2}.
       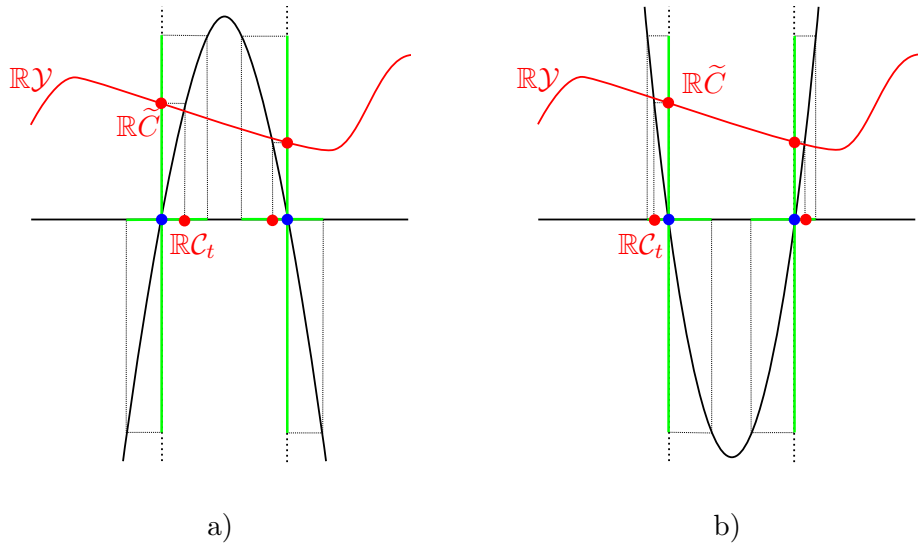
\begin{figure}[h!]
\centering
\begin{tabular}{ccc}
  \includegraphics[width=5cm, angle=0]{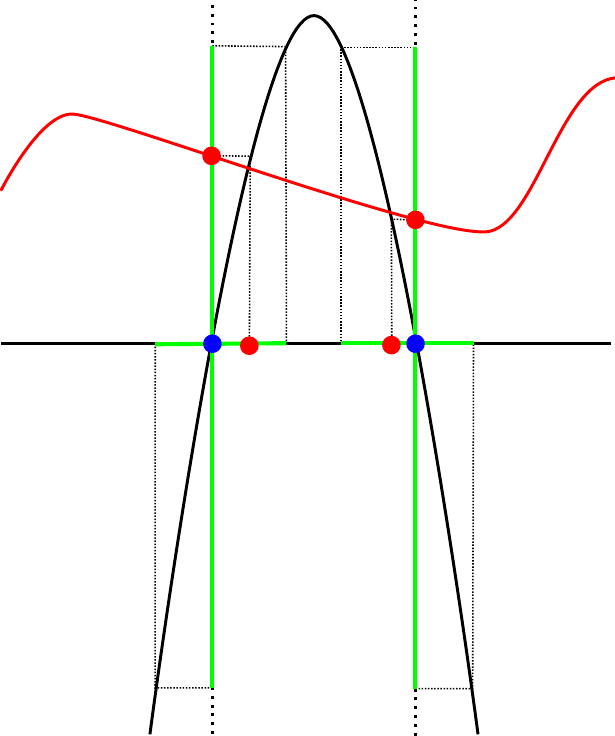}
   \put(-150,140){\textcolor{red}{$\RR\Y$}}
  \put(-110,122){\textcolor{red}{$\RR \widetilde C$}}
  \put(-90,78){\textcolor{red}{$\RR \C_t$}}
 &\hspace{5ex} &
  \includegraphics[width=5cm, angle=0]{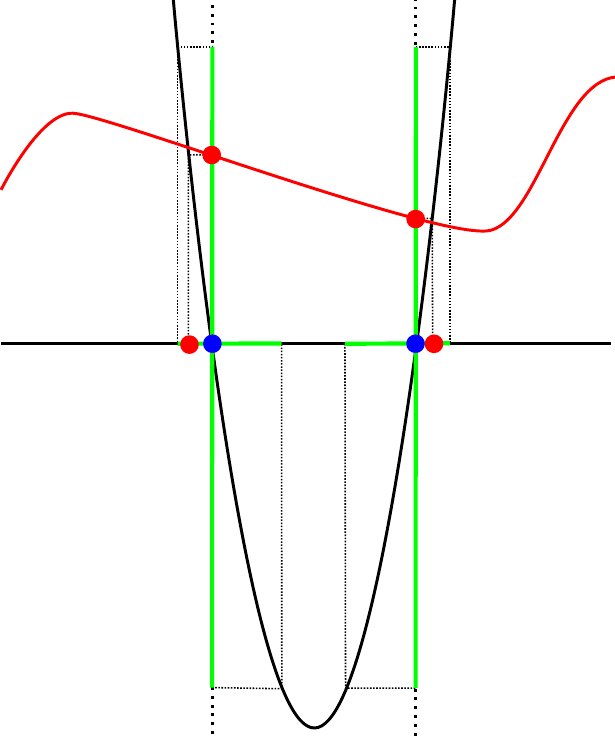}
    \put(-150,140){\textcolor{red}{$\RR\Y$}}
  \put(-88,138){\textcolor{red}{$\RR \widetilde C$}}
  \put(-112,78){\textcolor{red}{$\RR \C_t$}}
 \\
  \\ a) & &b) 
\end{tabular}
\caption{Deformation to normal cone and deformation of hypersurfaces
  ($k=1$ here).}
\label{fig:defnc2}
    \end{figure}
 Furthermore if  $t_1$ and $t_2$ are 
 two small enough real numbers of different signs, the two
 curves $\phi_{t_1}\circ s(\RR \C_{t_1})$ and
 $\phi_{t_2} (\RR \C_{t_2})$
 are isotopic in $\RP^2$, where
 \[
 \begin{array}{cccc}
   s:& V&\longrightarrow & V
   \\ & [x:y:z:w]&\longmapsto &[x:y:z:-w]
 \end{array}.
 \]

\medskip
Now let us go back to the particular setting of the proposition. 
To finish the proof, it is then enough to show that there exists such
real algebraic hypersurface $\Y$ in $V\times\CC$ such that the pair
$(N_0,\RR \widetilde C)$ is homeomorphic to the
pair $(N,\gamma_C)$.

The line bundle  $\mathcal O_{\CC P^1}(4)$ is the
complement of the point $[0:0:1]$ in the weighted projective plane
$W=\CP^2(1,1,4)$. The projection $\mathcal O_{\CC P^1}(4)\to\CP^1$ is
simply the linear projection
$W\setminus \{[0:0:1]\}\to\{w=0\}=\CP^1$. Let us fix a coordinate system
$[u:v:w]$ on $W$ (and hence also on  $\mathcal O_{\CC P^1}(4)$).
In the chart $v=1$,
a $k$-section in  $\mathcal O_{\CC P^1}(4)$ is defined by a polynomial
whose Newton polygon $\Delta$ is contained in
the triangle $\Delta_k$. If in addition
the $k$-section does not pass neither through $[1:0:0]$ nor $[0:1:0]$, 
 then the equality $\Delta=\Delta_k$ holds.
Conversely, the $k$-section $C$ in $\CC^2$, viewed as the chart
$v=1$ of $W$,
compactifies to
a $k$-section $\widetilde C$ in $\mathcal O_{\CC P^1}(4)$ that intersects
transversally the fiber with equation $v=0$, and the pair
$(\RR \mathcal O_{\CC P^1}(4),\RR\widetilde C)$ is homeomorphic to the
pair
$(N,\gamma_C)$.

The embedding
\[
\begin{array}{cccc}
  \psi: & W&\longrightarrow & V\times \CC
  \\ & [u:v:w]&\longmapsto & ([u^2:uv:v^2:w],0)
\end{array}
\]
has image $H_0$.
The closure $\widetilde C$ of $C$ in $W$ has equation
\[
  P(u,v,w)=\sum_{i_1+i_2+4j= 4k}a_{i_1,j}u^{i_1}v^{i_2}w^j.
  \]
If $ i_1+i_2+4j= 4k$, then $i_1$ is even if and only if $i_2$ is
even, so we can rewrite 
  \[
  P(u,v,w)=\sum_{i_1+i_2+2j= 2k}a_{2i_1,j}(u^2)^{i_1}(v^2)^{i_2}w^j
  + \sum_{i_1+i_2+1+2j= 2k}a_{2i_1+1,j}(u^2)^{i_1}(uv)(v^2)^{i_2}w^j.
  \]
  Hence if we define
  \[
  Q(x,y,z,w)=\sum_{i_1+i_2+2j= 2k}a_{2i_1,j}x^{i_1}z^{i_2}w^j
  + \sum_{i_1+i_2+1+2j= 2k}a_{2i_1+1,j}x^{i_1}yz^{i_2}w^j,
  \]
we see that the curve $\psi(\widetilde C)$ is defined  in
$V\times\CC$ by the system of equations
\[
\left\{
\begin{array}{ll}
  t=0
  \\ y^2-xz=0
  \\ Q(x,y,z,w)=0.
\end{array}
\right.
\]
Let $\Y$ be the hypersurface of $V\times\CC$ with equation
\[
 Q(x,y,z,w)=0.
\]
We have seen that $\Y\cap H_0=\widetilde C$, and this intersection is
transverse since $\widetilde C$ is non-singular.
The curve $\C_t=\Y\cap H_t$, with $t\ne 0$, is defined in $V$ by the system of equations
\[
\left\{
\begin{array}{ll}
tw- (y^2-xz)=0
  \\ Q(x,y,z,w)=0.
\end{array}
\right.
\]
Replacing $w$ with
$\frac{y^2-xz}{t}$ in the second equation and multiplying by $t^k$,
yields the equation
of $\C_t$ in $\CP^2$ stated in the proposition.
\end{proof}

\subsection{Controlling real bitangents of a perturbation of a multiple real conic}\label{sec:real bit}

In this section we tweak Proposition \ref{prop:conic2} by describing
the position of the resulting real algebraic curve 
with respect to \emph{all} its real
bitangents.
As before, we consider an even
integer $d=2k$, and $C$  a real $k$-section in $\CC^2$ defined by
the real polynomial
  \[
  P(u,w)=\sum_{i+4j\le 4k}a_{i,j}u^iw^j.
  \]
  We suppose in addition that the curve $C$ is in
  \emph{general position} with respect to the vertical pencil of lines, that is
 the  projection 
\[
\begin{array}{cccc}
\pi:&  C\subset\CC^2&\longrightarrow& \CC
  \\& (x,y) &\longmapsto & x
\end{array}
\]
admits  exactly $d(d-2)$ distinct fibers  $F_1,\ldots,F_{d(d-2)}$ tangent to $C$. 
In other words, all critical points of $\pi$ are non-degenerate and
have distinct critical values.

We need to introduce some notation before giving the main statement of
this section.
Recall the embedding
$\psi:W\to V\times\{0\}$ from the proof of Proposition
\ref{prop:conic2}, and
denote by $p_i$ the intersection of $\psi(F_i)$ with the conic $\C_0$.
Given $i\in\{1,\ldots,d(d-2)\}$,
we denote by $q_i$  the
tangency point of  $F_i$ and $C$, and by
$q_{i,1},\ldots, q_{i,k-2}$ the other points in $C\cap F_i$.
 If $q_i$ and $q_{i,l}$ are both real, we
say that    $q_i$ is \emph{above} or \emph{below} $q_{i,l}$
if the second coordinate of $q_i$ is greater or smaller that
the second coordinate of $q_{i,l}$ respectively.
We further choose $U_i$
 a small neighborhood in $V\times\CC$  of $\psi(q_i)$,
and  $U_{i,l}$ a small neighborhood in
$V\times \CC$ of $\psi(q_{i,l})$.
All neighborhoods $U_i$ and $U_{i,l}$ are chosen to be pairwise disjoint.
Finally given $i\ne j\in\{1,\ldots,d(d-2)\}$, we denote by $L_{i,j}$ the line
in $\CP^2$  through $p_i$ and $p_j$, and by $L_i$ the line
tangent to $\C_0$ at $p_i$.
Since a line in $\CP^2$ corresponds to a point in the dual plane $(\CP^2)^\vee$, the
notion of a limit of a sequence of lines in $\CP^2$ makes sense.

\begin{example}\label{ex:3sec2}
  Let us consider again the 3-section from Example \ref{ex:3sec}.
  We depicted in Figure \ref{fig:bit0}a two real tangent fibers $F_i$
  and $F_j$ of
  $C_3$, the corresponding lines $L_{i,j}$, $L_i$, and $L_j$ in $\CP^2$ are
  depicted in Figure \ref{fig:bit0}b. 
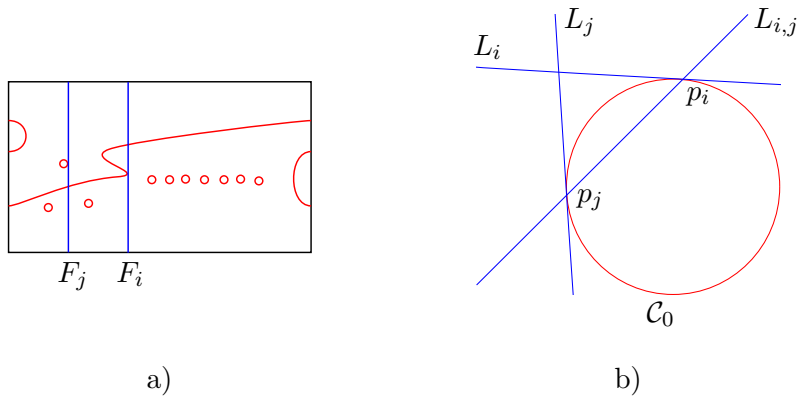
\begin{figure}[h!]
\centering
\begin{tabular}{ccc}
  \includegraphics[width=4cm, angle=0]{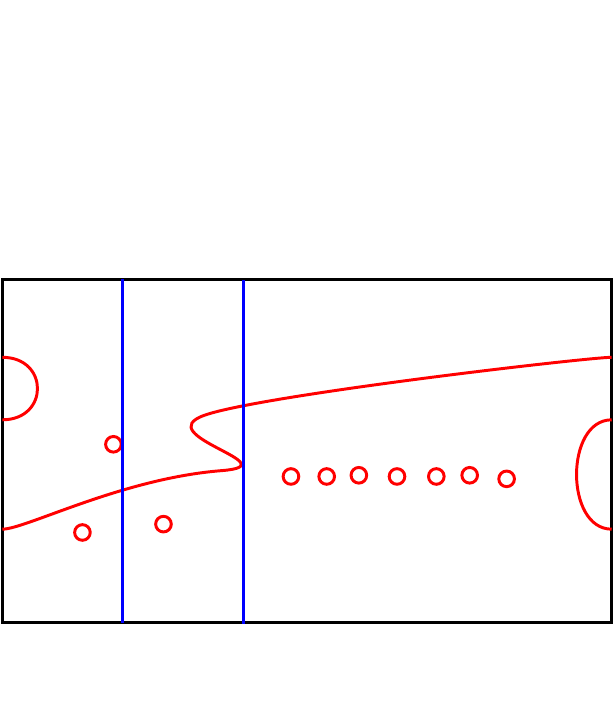}
  \put(-95,5){$F_j$}
   \put(-73,5){$F_i$}
 & \hspace{8ex} &
  \includegraphics[width=4cm, angle=0]{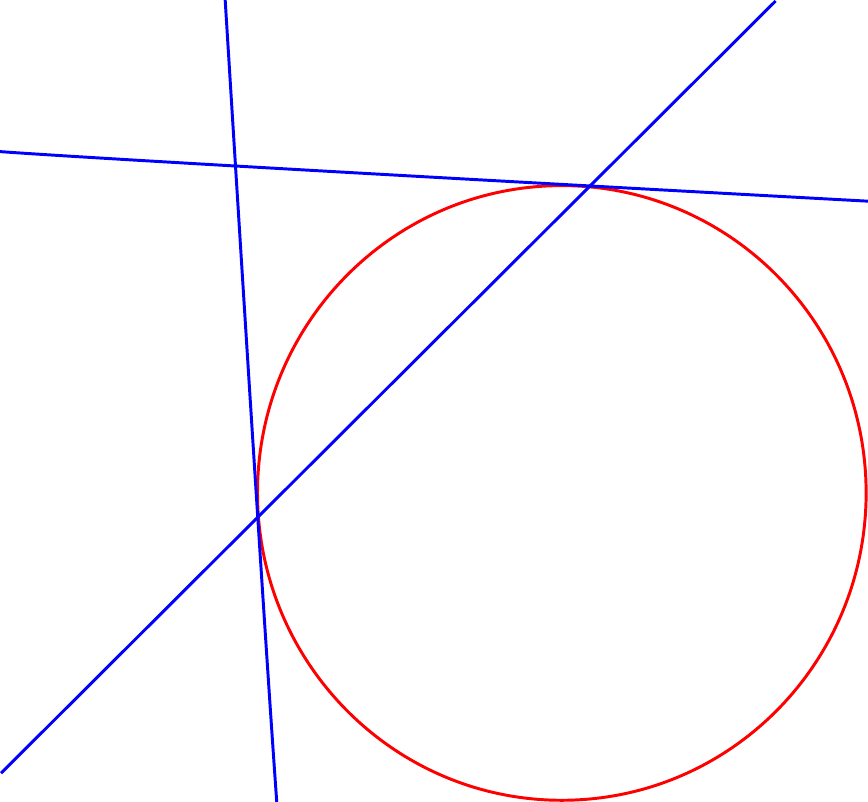}
   \put(-50,-10){$\C_0$}
   \put(-35,73){$p_i$}
   \put(-76,35){$p_j$}
      \put(-115,90){$L_i$}
  \put(-81,100){$L_j$}
   \put(-10,100){$L_{i,j}$}
\\
  \\ a) && b)
\end{tabular}
\caption{The $3$-section $C_3$ and two of its tangent fibers.}
\label{fig:bit0}
    \end{figure}
\end{example}

As in Proposition \ref{prop:conic2}, we consider  the 
  real algebraic curve  $\C_t$ of degree $d$
  defined by the homogeneous equation
  \[
  \sum_{i_1+i_2+2j= 2k} t^{k-j} a_{2i_1,j} x^{i_1}(y^2-xz)^jz^{i_2}  +
   \sum_{i_1+i_2+2j+1= 2k} t^{k-j} a_{2i+1,j} x^{i_1}y(y^2-xz)^jz^{i_2} =0.
  \]

  \begin{theo}\label{thm:bit}
A line $L$ in $\CP^2$ is the limit of a bitangent of $\C_t$ as $t$
tends to $0$ if and only if $L=L_i$ or $L=L_{i,j}$.
Furthermore:
\begin{enumerate}
\item $L_{i,j}$ is the limit of exactly
  1 bitangent $L_{i,j,t}$ of $\C_t$.
 The line $L_{i,j,t}$  is tangent to $\C_t$ at one point in $U_{i}$
 and one in $U_j$.
 In particular $L_{i,j,t}$ is real if and only if
  $p_i$ and $p_j$ are both real or complex conjugated, and its
 position with respect to $\RR\C_t$ is determined by $\RR C$, see
 Figure
 \ref{fig:bit}.
       \begin{figure}[h!]
\centering
\begin{tabular}{ccc}
  \includegraphics[width=6cm, angle=0]{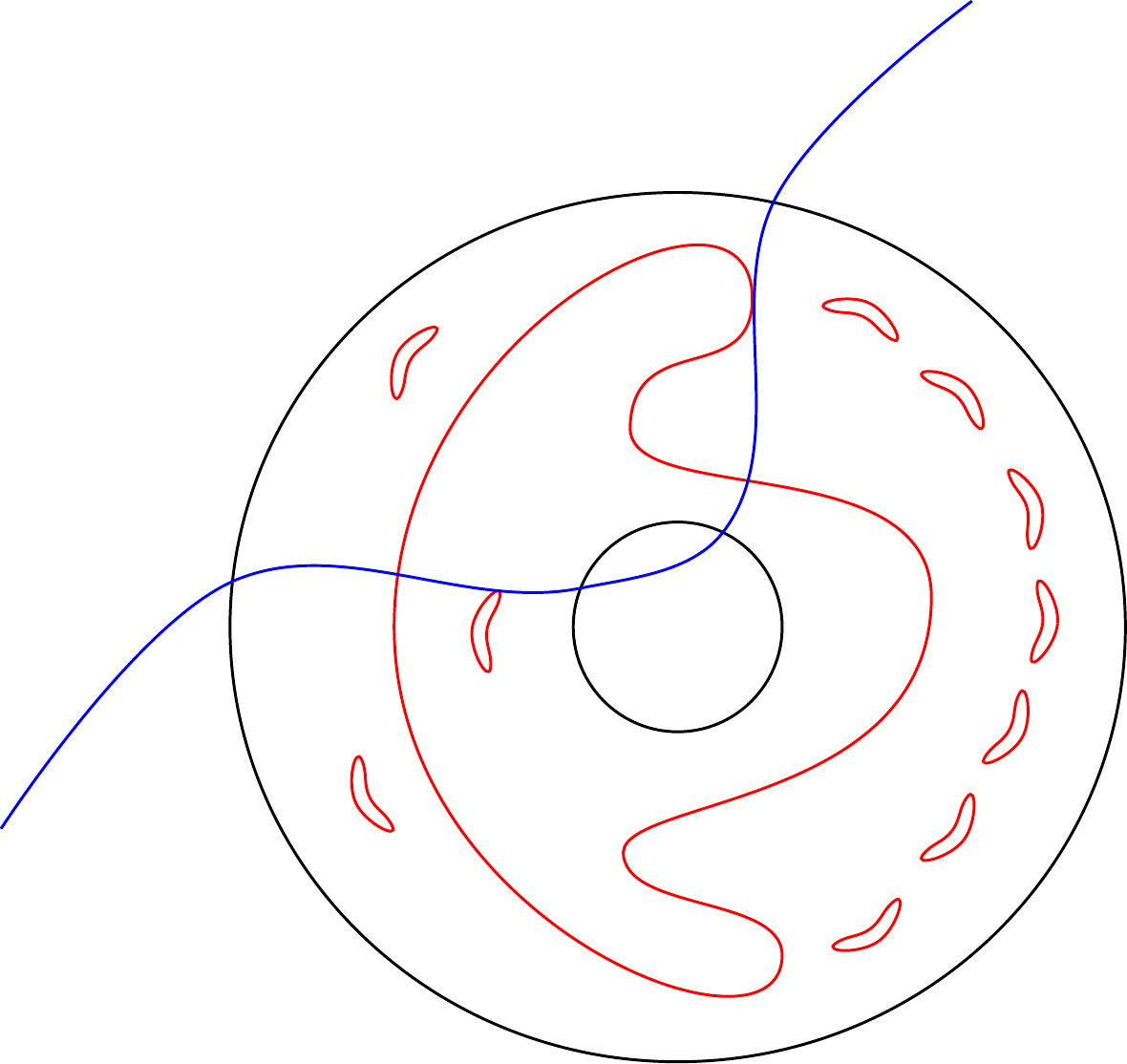}
  &\hspace{5ex}&
  \includegraphics[width=6cm, angle=0]{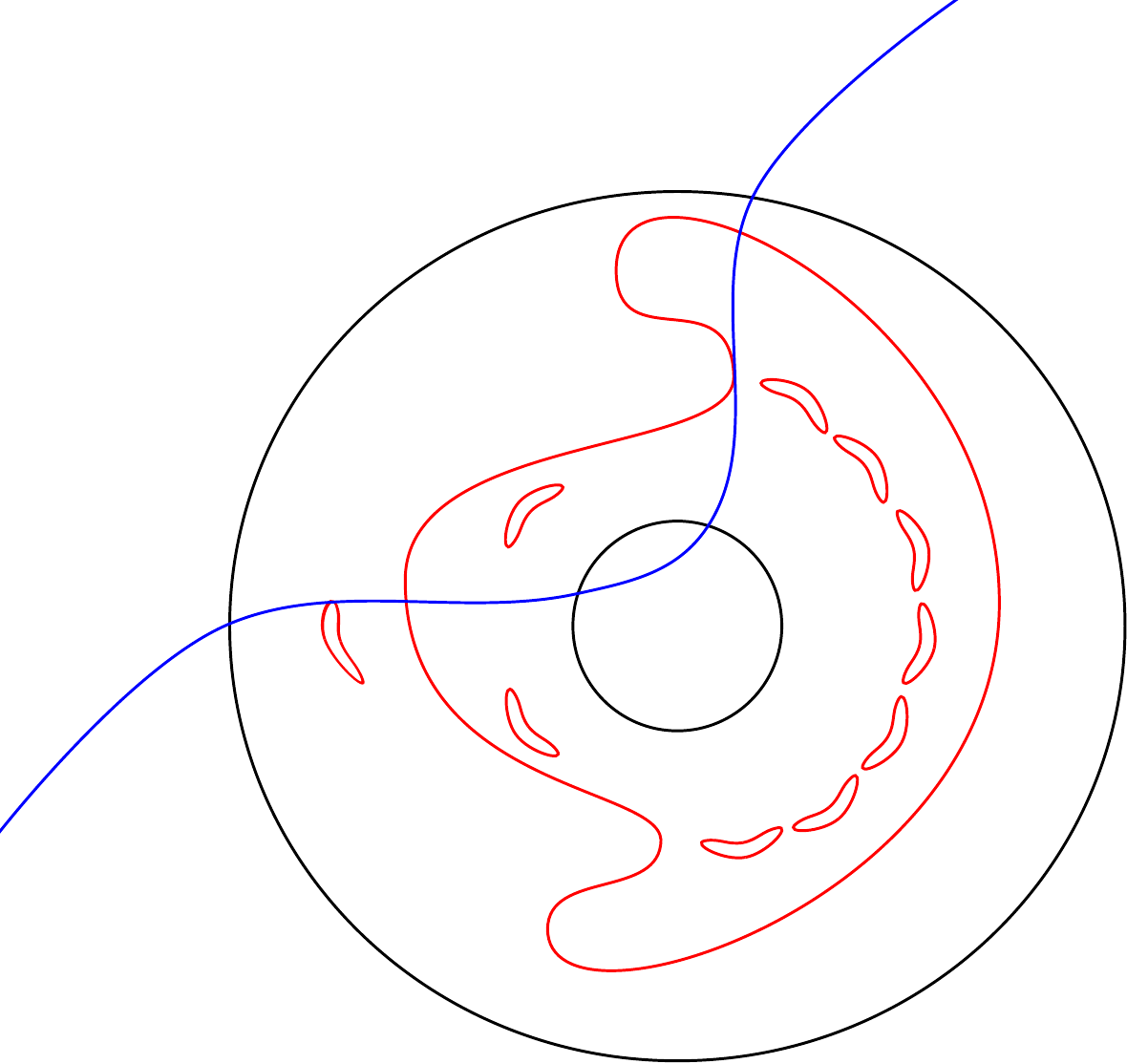}
  \\
  \\ $t<0$ && $t>0$
\end{tabular}
\caption{Deformation of a line $L_{i,j}$ to a bitangent of $\RR \C_t$,
 Example \ref{ex:3sec2} continued. In order to make the picture
 intelligible without a microscope, the small neighborhood of
 $\RR \C_0$  containing $\RR \C_t$
 has been dramatically stretched. As a result the deformation of $L_i$
does not look straight anymore; nevertheless the arrangement
of $\RR\C_t$ and $\RR L_{i,j,t}$ in $\RP^2$ is isotopic to this picture.}
\label{fig:bit}
    \end{figure}

\item
  $L_i$ is the limit of exactly
  $2(k-2)$ bitangents  of $\C_t$. For each $l\in\{1,\ldots, k-2\}$,
  there exist exactly 2 bitangents $L^1_{i,l,t}$ and $L^2_{i,l,t}$
  of $\C_t$ that are  tangent to $\C_t$  at one point in $U_{i}$
 and one  in $U_{i,l}$, and converging to $L_i$.
The bitangents  $L^1_{i,l,t}$ and $L^2_{i,l,t}$ are  both real
 if and only if
both $q_i$ and $q_{i,l}$ are real and
\begin{enumerate}
\item $t>0$ and $q_i$ is above $q_{i,l}$;
\item   $t<0$ and $q_i$ is below $q_{i,l}$;
\end{enumerate}
otherwise, both are non-real.
When  $L^1_{i,l,t}$ and $L^2_{i,l,t}$ are real, their
position with respect to $\RR\C_t$ is determined by $\RR C$, see
 Figure
 \ref{fig:bit2}.
       \begin{figure}[h!]
\centering
\begin{tabular}{ccc}
  \includegraphics[width=6cm, angle=0]{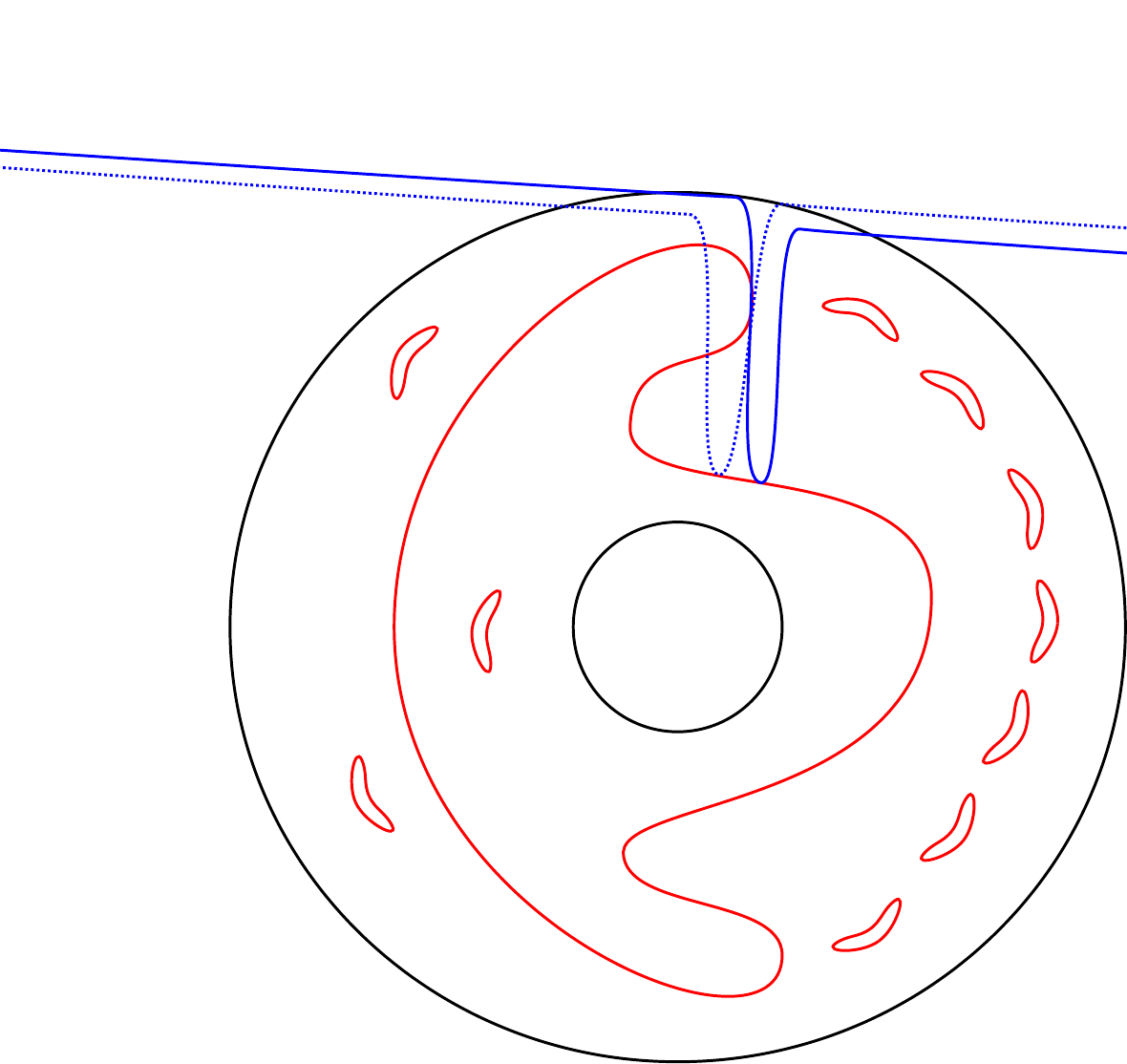}
   \put(-170,145){$L^1_{j,1,t}$}
   \put(-170,122){$L^2_{j,1,t}$}
 &\hspace{5ex}&
  \includegraphics[width=6cm, angle=0]{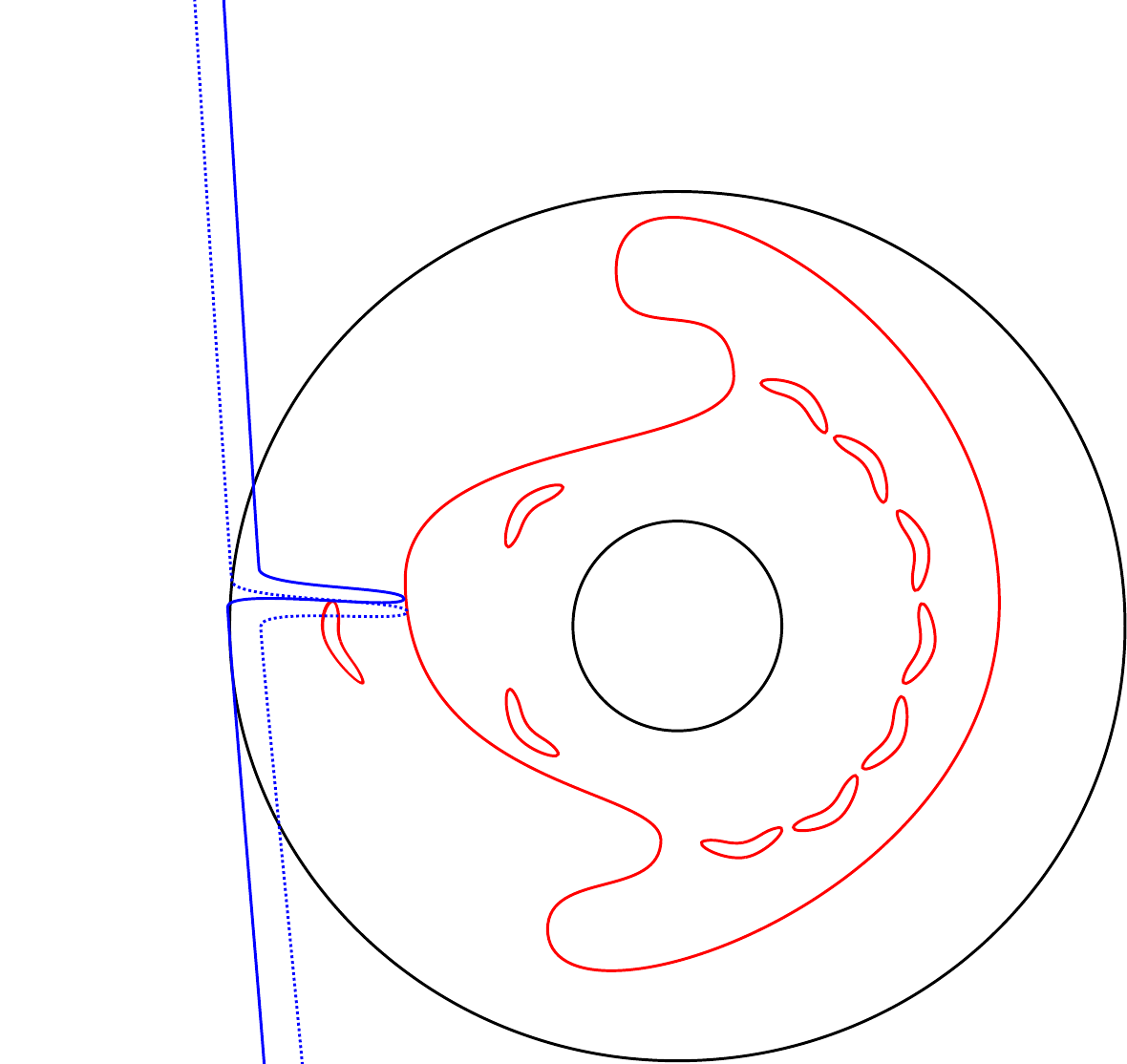}
    \put(-167,145){$L^2_{i,1,t}$}
   \put(-128,145){$L^1_{i,1,t}$}
  \\
  \\  $t<0$ && $t>0$
\end{tabular}
\caption{Deformation of a line $L_{i}$ to real bitangents of $\RR \C_t$,
 Example \ref{ex:3sec2} continued.}
\label{fig:bit2}
    \end{figure}
\end{enumerate}
  \end{theo}




  

  \begin{proof}
The proof goes by locating (real) nodes of the curve $\C_t^\vee$ dual
to $\C_t$. Recall that  nodes of $\C_t^\vee$ correspond to
bitangents of $\C_t$, and that cusps of $\C_t^\vee$ correspond to
flexes of $\C_t$. Furthermore if $\C_t^\vee$ has exactly
$\frac12 d(d-2)(d-3)(d+3)$ nodes, then $\C_t^\vee$ has only nodes and
cusps as its singularities.
To locate the nodes of $\C_t^\vee$, we write local parameterizations
of $\C_t^\vee$ that are dual of local parameterizations
of $\C_t$.
    
    We define
    $\Omega=\{z\in \CC\  \ |z|<1\}$, and we denote by 
    $\pi:D(4)\to\CP^1$ the $\Omega$-bundle which is the restriction of
    $\pi:\O_{\CP^1}(4)\to \CP^1$ for some choice of a Hermitian metric on
    $\O_{\CP^1}(4)$.  Up to rescaling by the action of $\CC^*$ on
    $\O_{\CP^1}(4)$, we may assume that
    that $C\subset D(4)$.
The projection $\pi=\pi_0: D(4)\to \CP^1$ deforms to
a disk bundle $\pi_t:N_t\to \C_0$ for $t$ small enough,
where $N_t$ is the tubular
neighborhood of $\C_0$ from the proof of Proposition
\ref{prop:conic2}.
Let $\nu\in C\subset \O_{\CP^1}(4)$
 and $\rho_{\nu,0}: \Omega\to C$   be a local holomorphic parameterization
of $C$ at $\nu$, with  $\rho_{\nu,0}(0)=\nu$.
We denote by $\nu_t$ a deformation of $\nu$ in $\C_t$, and
by  $\rho_{\nu,t}: \Omega\to \C_t$ a deformation of $\rho_{\nu,0}$ such that
$\rho_{\nu,t}(0)=\nu_t$. 
We define $\theta_{\nu,t}=\pi_t\circ\rho_{\nu,t} :\Omega\to \C_0$. 
Without loss of generality, we may further assume
that
\begin{enumerate}
  \item the fiber of $\O_{\CP^1}(4)$ passing through $\nu$
    intersects $\CP^2$ at the point $[0:0:1]$.
\item
if $\nu$ is  a critical  point of $\pi|_{C}:C\to\CP^1$,
  then    $0$ is a critical point of $\theta_{\nu,t}$
  with critical value $[0:0:1]$ for $t$ small enough.  
\end{enumerate}

Hence in the chart $z=1$, the map $\theta_{\nu,t}:\Omega\to \CP^2$ can be
written as
\[
\theta_{\nu,t}(u)=(u^{2n_0}+t\alpha_t(u), u^{n_0}+ t\beta_t(u))
\]
where  $\alpha$ and
$\beta$ are two holomorphic functions in $(t,u)$, and $n_0$ is the
ramification index of $\pi_0$ at $\nu$. Recall that by assumption on
$C$, the integer $n_0$ can only take the values $1$ and $2$.

\bigskip
Let $[a:b:c]$ be the coordinate system on $(\CP^2)^\vee$ dual to the
 chosen  coordinate system $[x:y:z]$ on $\CP^2$. That is, the point
 $[a:b:c]\in(\CP^2)^\vee$ corresponds to the line in $\CP^2$ with
 equation
 \[
 ax+by+cz=0.
 \]
 One computes easily that the conic  $\C_0^\vee$ in   $(\CP^2)^\vee$ dual to the
 conic $\C_0$ in $\CP^2$ has equation
 \[
 b^2-4ac=0.
 \]

We denote by $\theta^\vee_{\nu,t}$ the dual of the above map $\theta_{\nu,t}$, that
is to say the map that associates
to $u$ the line in $(\CP^2)^\vee$ tangent to $\theta_{\nu,t}(u)$.
In what follows, we denote by $O_\H(t^ku^l)$ a
function defined on a neighborhood $U$ of $0$ in $\CC^2$ and
of the form $t^ku^l\gamma(t,u)$ with $\gamma$ holomorphic on $U$.
The line in $\CC^2$ passing through the point $(x_0,y_0)$ with
directing vector $(v_1,v_2)$ has equation
\[
v_2x - v_1y -v_2x_0 + v_1y_0.
\]
Hence we deduce from
\[
\theta_{\nu,t}'(u)=\left(
2n_0u^{2n_0-1}+t\alpha'_t(u) , n_0u^{n_0-1}+t\beta'_t(u) 
\right)
\]
that
\begin{align*}
  \theta_{\nu,t}^\vee(u)&=
  \left[
    n_0u^{n_0-1}+t\beta'_t(u) : -2n_0u^{2n_0-1}-t\alpha'_t(u):\right.
 \\ &\qquad\qquad \left.  -(   n_0u^{n_0-1}+t\beta'_t(u))(u^{2n_0}+t\alpha_t(u))
 +(2n_0u^{2n_0-1}+t\alpha_t'(u))(u^{n_0}+ t\beta_t(u))
    \right]
\end{align*}

\noindent {\bf Case 1: $n_0=1$.} One has then
\begin{align*}
  \theta_{\nu,t}^\vee(u)&=
   \left[
    1+O_\H(t) : -2u-t\alpha'_t(u):   u^{2}-t\alpha_t(u) +O_\H(tu) +O_\H(t^2)    \right].
\end{align*}

In the chart $a=1$ of  $(\CP^2)^\vee$, this gives
\begin{align*}
  \theta_{\nu,t}^\vee(u)&=\left(
  \frac{-2u-t\alpha'_t(u)}{1+O_\H(t)},
   \frac{ u^{2}-t\alpha_t(0) + O_\H(tu) +O_\H(t^2) }{1+O_\H(t)}
   \right)
   \\
   \\&=\left(
  -2u-t\alpha'_t(0) +O_\H(tu) +O_\H(t^2) ,
   u^{2}-t\alpha_t(0) +O_\H(tu)  +O_\H(t^2) 
   \right) .
\end{align*}
In particular we find that $\theta_{\nu,t}^\vee$ converges when $t$ goes to
$0$ to the local parameterization $u\mapsto (-2u,u^2)$ at $(0,0)$ of
$\C_0^\vee$, in accordance with the fact that $\pi_{0}|_C:C\to\CP^1$ is a local
biholomorphism  at $\nu$.

\noindent {\bf Case 2: $n_0=2$.}

By assumption $(ii)$ above, we have $\beta_t(0)=\beta'_t(0)=0$. Hence
$\beta_t(u)=O_\H(u)$ and $\beta'_t(u)=O_\H(u)$. Hence one has
\begin{align*}
  \theta_{\nu,t}^\vee(u)&=
   \left[
    2u+O_\H(tu) : -4u^{3}-t\alpha'_t(u):  2 u^{5}-2tu\alpha_t(u) +O_\H(tu^2) +O_\H(t^2u)    \right].
\end{align*}
In the chart $a=1$ of  $(\CP^2)^\vee$, one obtains
\begin{align*}
  \theta_{\nu,t}^\vee(u)&=\left(
  \frac{ -4u^{3}-t\alpha'_t(u)}{ 2u+O_\H(tu)  },
  \frac{2 u^{5}-2tu\alpha_t(0) +O_\H(tu^2)  +O_\H(t^2u)  }{2u+O_\H(tu) }
  \right)
  \\
  \\
  &= \left(
  \frac{ -4u^{3}-t\alpha'_t(0)+O_\H(tu^3)+O_\H(t^2)}{ 2u },
   u^{4}-t\alpha_t(0) +O_\H(tu) +O_\H(t^2)
  \right)
\end{align*}
Hence in a neighborhood of $0$ in $(\CC^*)^2$, the map
$\theta_{\nu,t}^\vee(u)$ is a small perturbation of the map
\begin{align*}
  \widetilde \theta_{\nu,t}^\vee(u)&=\left(
  \frac{ -4u^{3}-t\alpha'_t(0)}{ 2u },
   u^{4}-t\alpha_t(0)
  \right)
\end{align*}
Note that when $t$ goes to $0$, the image of the map
$\theta_{\nu,t}^\vee(u)$ converges to the union of the line $c=0$ with a
neighborhood of $[1:0:0]$ in $\C_0^\vee$ covered twice.

Alltogether, this proves the following: the limit of $\C_t^\vee$ as
$t$ goes to $0$ is the divisor
\[
 \frac{d}{2} \C_0^\vee + \widehat L_1+\ldots +  \widehat L_{d(d-2)}
\]
in $\big|\O_{(\CP^2)^\vee}(d(d-1))\big|$,
where $ \widehat L_i$ is the tangent line of $\C_0^\vee$ at
 $L_i^\vee$. Furthermore if $i\ne j$, the intersection point of
 $ \widehat L_i$ and $ \widehat L_j$ deforms to a node of $\C_t^\vee$.
 Since the point $ \widehat L_i\cap \widehat L_j$ in $(\CP^2)^\vee$
 corresponds to the line $L_{i,j}$ in $\CP^2$, this proves part
 $(i)$ of the theorem.

 \bigskip
To prove part $(ii)$ of the theorem, choose
 $\mu\neq\nu_0 $ another point of $C$ on the same fiber than
$\nu$. By assumption on $C$, the point  $\mu$ is not a critical point
of $\pi_{0|C}$, and 
we have
\begin{align*}
  \theta_{\mu,t}^\vee(u)&=\left(
  -2u-t\widehat\alpha'_t(0) +O_\H(tu) +O_\H(t^2) ,
   u^{2}-t\widehat\alpha_t(0) +O_\H(tu)  +O_\H(t^2) 
   \right) .
\end{align*}
with $\widetilde\alpha$ a holomorphic function in $(t,u)$.
We now show the following reformulation of
the first part of statement $(ii)$ of the theorem:
the equation
\[
\theta_{\nu,t}^\vee(u)=\theta_{\mu,t}^\vee(v)
\]
has exactly two solutions that
converge to $0$ as $t$ goes to $0$.
To that aim, it is enough
  to count the
number of small but no-null solutions of
the equation
\begin{equation}\label{eq:pt2}
  \widetilde\theta_{\nu,t}^\vee(u)=\widetilde\theta_{\mu,t}^\vee(v),
\end{equation}
where
\begin{align*}
  \widetilde\theta_{\mu,t}^\vee(v)&=\left(
  -2v-t\widehat\alpha'_t(0),
   v^{2}-t\widehat\alpha_t(0)
   \right) .
\end{align*}
The first coordinate gives
\[
v=\frac{ 4u^{3}+t\alpha'_t(0)-2tu\widehat\alpha'_t(0)}{4u},
\]
which plugged into the second coordinate gives after simplification
\[
\frac {tR_t(u)}{16u^6}=0,
\]
with
\[
R_t(u)= -16\widehat\alpha'_t(0)u^4  + 8\alpha'_t(0)u^3
+(4t\widehat\alpha'^2_t(0) -16\widehat\alpha_t(0)
  +16\alpha_t(0))u^2
  -4t\alpha'_t(0)\widehat\alpha'_t(0)u+
t\alpha'^2_t(0)  
  \]
  Since $0$ is a double root of
  \[
  R_0(u)=-16\widehat\alpha'_0(0)u^4  + 8\alpha'_0(0)u^3
+(-16\widehat\alpha_0(0)
  +16\alpha_0(0))u^2,
 \]
 the polynomial $R_t(u)$ has two roots in an arbitrary small
 (depending on $t$)
 neighborhood of the origin, which proves the claim.

 Furthermore note that, since the coefficients of $R_t(u)$ of degree 1
 and 0 are multiples of $t$ and that other coefficients have order 0
 in $t$ at 0, the discriminant of $R_t(u)$ has order 1
 in $t$ at 0. As a consequence  if $C$ is a real algebraic curve,
 then 
 the two small roots of polynomial $R_t(u)$ for $t$ real and small
 enough are either real or complex conjugated depending on the sign of
 $t$. For topological reasons, when real
 the line $L_i$ cannot deform to a real bitangent of $\C_t$ if
 \begin{enumerate}
\item $t<0$ and $q_i$ is above $q_{i,l}$;
\item   $t>0$ and $q_i$ is below $q_{i,l}$.
\end{enumerate}
 This proves the second part of part $(ii)$ of the theorem.
 To get a feeling of what is going on, we depicted $\C_t^\vee$
 on Figure \ref{fig:tgt} in a neighborhood of $[1:0:0]$.
    \begin{figure}[h!]
\centering
\begin{tabular}{c}
  \includegraphics[width=6cm, angle=0]{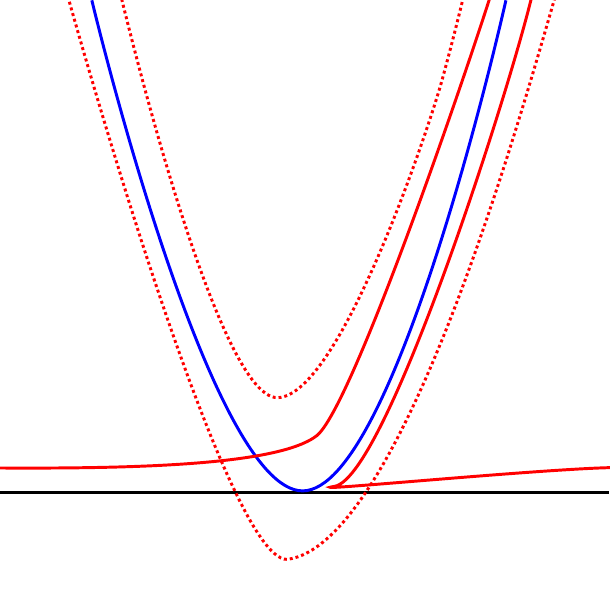}
\end{tabular}
\caption{The dual curve $\C_t^\vee$, in red,
  in a neighborhood of $L_i^\vee$;  the continuous part
is the branch parameterized by $\theta_{p_i,t}^\vee$, while the dotted
parts
are branches parameterized by $\theta_{q_{i,j},t}^\vee$; the blue
part is $\C_0^\vee$.}
\label{fig:tgt}
    \end{figure}

 \bigskip
 To end the proof of the theorem, it remains to check that $\C_t$ has
 no other complex bitangents. 
This is done by counting how many bitangents we already identified:
    \begin{align*}
    \binom{d(d-2)}{2} + 2(\frac{d}{2}-2)d(d-2)&=
 \frac{d(d-2)(d(d-2)-1)}{2}+ (d-4)d(d-2)
      \\ &= \frac12 d(d-2)(d-3)(d+3),
    \end{align*}
which is the total number of bitangents of  $\C_t$.
\end{proof}

\begin{example}
   Let us consider again the 3-section $C_3$ from Example
   \ref{ex:3sec}. Applying Theorem \ref{thm:bit} to $C_3$, one obtains
   two maximal real sextics: the Harnack sextic (Figure \ref{fig:3sec}b) and
   the Hilbert sextic  (Figure \ref{fig:3sec}c).
Thanks to Theorem \ref{thm:bit}, one easily sees that the Harnack
sextic has exactly $316$ real bitangents distributed as follows:
\begin{itemize}
\item 1 real bitangent to  each empty oval (10);
  \item 10 real bitangents to the non-empty oval (10);
 \item  4 real bitangents to each pair of empty oval has (180);
\item 8 real bitangents  to the non-empty oval and the  odd oval (8);
\item 12 real bitangents  to the non-empty oval and each empty even oval (108);
\end{itemize}
and that the Hilbert
  sextic has exactly $284$ real bitangents distributed as follows:
 \begin{itemize}
\item 1 real bitangent to  each empty oval (10);
  \item 10 real bitangents to the non-empty oval (10);
 \item  4 real bitangents to each pair of empty oval has (180);
\item 8 real bitangents  to the non-empty oval and each  odd oval (72);
\item 12 real bitangents  to the non-empty oval and the empty even oval (12).
\end{itemize}
\end{example}

\section{Constructions}\label{sec:constr}

In this section, we apply Theorem \ref{thm:bit} to perform the constructions
announced in Section \ref{sec:intro constr}, and to prove Proposition
\ref{prop:optimality}.

\subsection{Constructions in degree 6}
The proof of Theorem \ref{thm:sextics} relies on the following
cut-and-glue procedure: start with the
 $\mathcal L$-scheme $\mathcal L_0$ depicted in Figure \ref{fig:viro}a, and apply
recursively the  operation described below to construct a
$\mathcal L$-scheme $\mathcal L_{i+1}$ out of $\mathcal L_i$:
\begin{figure}[h!]
\centering
\begin{tabular}{ccccc}
  \includegraphics[width=3cm, angle=0]{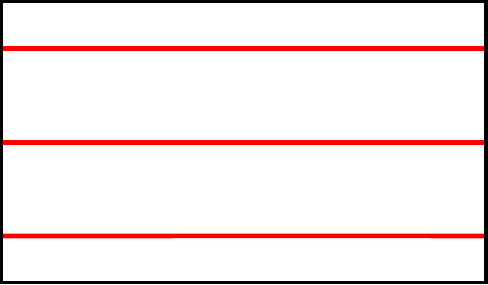}
  &\hspace{5ex}&
 \includegraphics[width=3cm, angle=0]{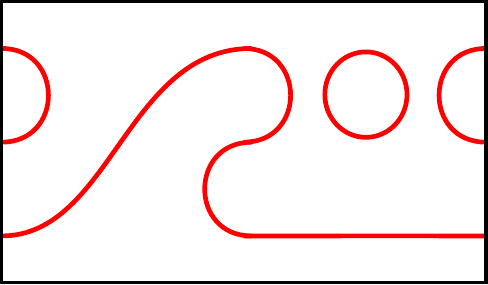}
 &\hspace{5ex}&
 \includegraphics[width=3cm, angle=0]{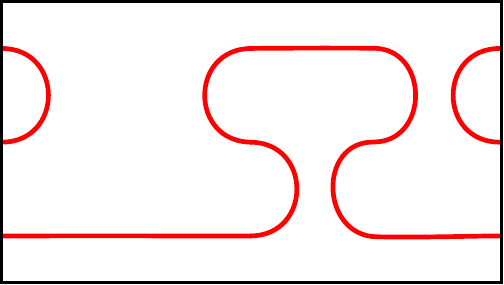}
 \\
  \\a) && b) && c)
\end{tabular}
\caption{Building blocks to construct 3-sections.}
\label{fig:viro}
    \end{figure}
 \begin{enumerate}
 \item choose a vertical line $l$ on $\mathcal L_i$
   intersecting  $\mathcal L_i$ in three distinct points;
 \item cut    $\mathcal L_i$ along the line $l$, and
 flip the left hand side by the orthogonal symmetry with respect to
 the
 abscissa axis;
\item the  $\mathcal L$-scheme $\mathcal L_{i+1}$ is obtained by
  the gluing the left (respectively right) hand side 
  of this cut-and-flipped $\mathcal L_{i}$
to the left
  (respectively right) side of one of the $\mathcal L$-scheme depicted
in Figure  \ref{fig:viro}b, Figure
  \ref{fig:viro}c, or their image under the orthogonal symmetry with respect to
the origin or any coordinate axis.
 \end{enumerate}

 \begin{example}\label{exa:3sec gud}
 We depict an example of such recursive
construction in Figure \ref{fig:gudkov}. At each step, the chosen line $l$ is
the dotted line, and the grey zone corresponds to the
cut-and-flipped $\mathcal L_{i-1}$.
 \begin{figure}[h!]
\centering
\begin{tabular}{ccccc}
 \includegraphics[width=2.3cm, angle=0]{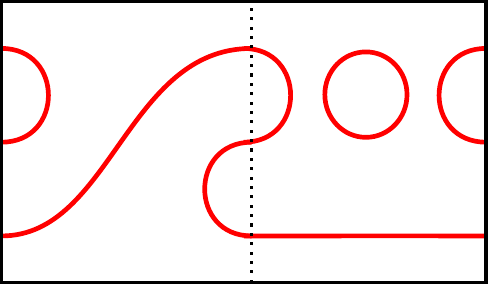}
 & &
 \includegraphics[width=4.6cm, angle=0]{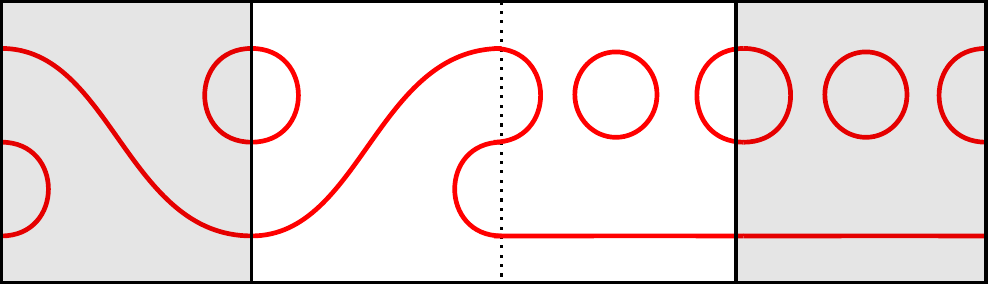}
  & &
 \includegraphics[width=7.2cm, angle=0]{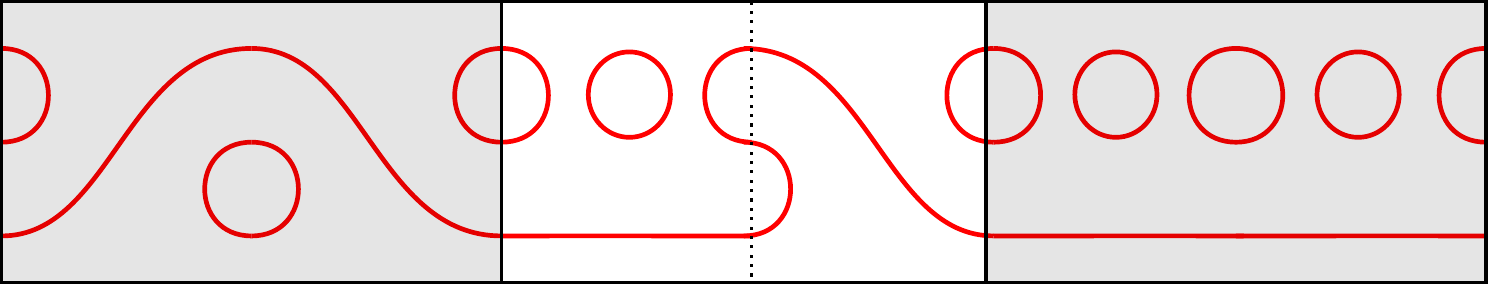}
\\ $\mathcal L_1$ &&$\mathcal L_2$  && $\mathcal L_3$
\\
\end{tabular}
\begin{tabular}{c}
  \\ 
 \includegraphics[width=10cm, angle=0]{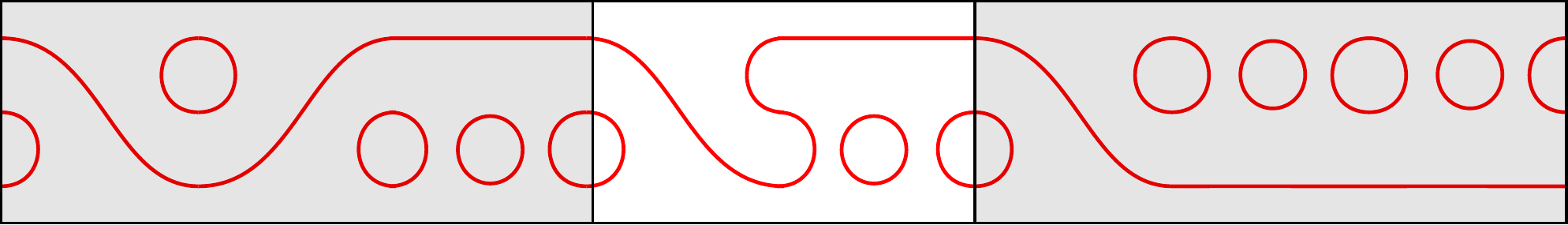}
 
  \\ $\mathcal L_4$
\end{tabular}
\caption{A 3-section deforming to the Gudkov sextic  $\langle 5 \sqcup 1  \langle 5\rangle\rangle$.}
\label{fig:gudkov}
    \end{figure}
\end{example}

The proof of next lemma is an
elementary application 
of Viro Patchworking Theorem (see for example \cite{V1}).
 \begin{lem}\label{lem:viro}
   Any $\mathcal L$-scheme $\mathcal L_{4}$ constructed by the above
   cut-and-glue procedure is realized by a real algebraic 3-section.
\hfill $\square$
 \end{lem}

 \begin{remark}
 Lemma \ref{lem:viro} is also an immediate consequence of Orevkov's
 method to construct real trigonal curves via Dessins d'Enfants
 \cite{O3}. However Viro's Patchworking Theorem is constructive,
 whereas Orevkov's
 method, based on Riemann Existence Theorem, is not.
 \end{remark}
 

 Recall that the list of real schemes realized by a 
 real algebraic plane
 sextics is given in Figure \ref{fig:rc sextic}.

\begin{coro}\label{cor:max tang 3-sec}
  Let  $\mathcal R\ne \emptyset, h$ be a real scheme realized by a 
 real algebraic plane
 sextic.  Then there exists a $3$-section $C$ such that
 \begin{itemize}
 \item  one of the two sextics obtained by 
   applying Proposition \ref{prop:conic2} to $C$ has real scheme
   $\mathcal R$;
   \item all of the $24$  vertical lines in
     $\CC^2$ that are tangent to $C$ are real.
 \end{itemize}
\end{coro}
\begin{proof}
  The proof  for real schemes
  $\langle a \sqcup 1  \langle b \rangle\rangle$ with
  $a+b=10$ is provided by Examples
   \ref{ex:3sec} and \ref{exa:3sec gud}.
   Any $3$-section realizing an $\mathcal L$-scheme $\mathcal L_4$
   obtained by the above cut-and-glue procedure has $24$ real tangent
   vertical lines. Hence 
  we are left to realize all the remaining real schemes using Proposition
  \ref{prop:conic2} and Lemma \ref{lem:viro}.
  All possible $\mathcal L$-schemes $\mathcal L_4$
obtained from 
  $\mathcal L_0$ and  four distinct vertical lines in $\mathcal L_0$
 realize the following real schemes:
  \begin{itemize}
  \item $\langle a\rangle$ with $a\le 9$;
  \item  $\langle a \sqcup 1  \langle b \rangle\rangle$ with $a+b\le 5$;
  \item  $\langle 5 \sqcup 1  \langle 1 \rangle\rangle$ and
    $\langle 1 \sqcup 1  \langle 5 \rangle\rangle$.
  \end{itemize}
 All possible $\mathcal L$-schemes $\mathcal L_4$
obtained from  the  $\mathcal L$-scheme
  depicted in Figure \ref{fig:sextic}a
 \begin{figure}[h!]
\centering
\begin{tabular}{ccccc}
  \includegraphics[width=3cm, angle=0]{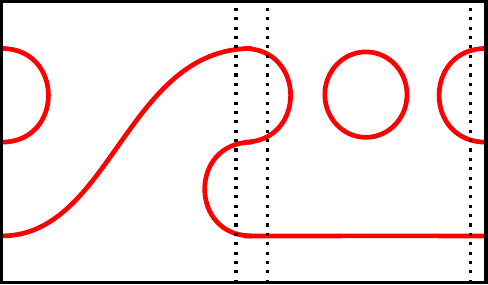}
  &\hspace{5ex}&
 \includegraphics[width=3cm, angle=0]{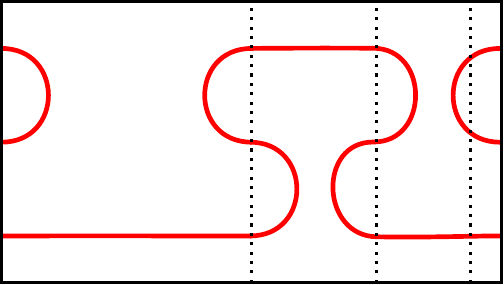}
 &\hspace{5ex}&
 \includegraphics[width=6cm, angle=0]{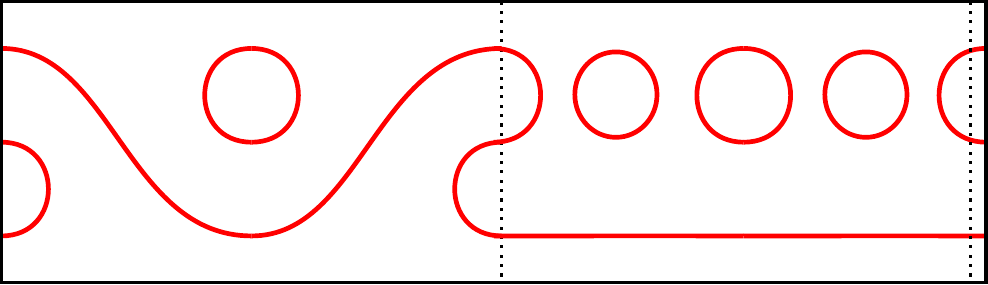}
 \\
  \\a) && b) && c)
\end{tabular}
\caption{Construction of 3-sections.}
\label{fig:sextic}
    \end{figure}
 and  the three dotted vertical lines  realize 
   all the remaining real schemes except the following ones:
  \begin{itemize}
  \item   real schemes with 10 ovals;
  \item  $\langle 8-b \sqcup 1  \langle b \rangle\rangle$ and
    $\langle b \sqcup 1  \langle 8-b \rangle\rangle$ with $b=1,3$.    
  \end{itemize}
One constructs  an $\mathcal L$-scheme $\mathcal L_4$
 from  the  $\mathcal L$-scheme
  depicted in Figure   \ref{fig:sextic}b
  and  the three dotted vertical lines, that realizes
  the  real schemes $\langle 10\rangle$
  and $\langle  1  \langle 9 \rangle\rangle$.
  All possible $\mathcal L$-schemes $\mathcal L_4$
obtained from  the  $\mathcal L$-scheme
depicted in Figure \ref{fig:sextic}c
and  the two dotted vertical lines realize
 the additional real schemes 
  and $\langle 1 \sqcup   1  \langle b \rangle\rangle$ and
  $\langle b \sqcup   1  \langle 1 \rangle\rangle$ with
  $b=7,8$. By an easy modification of the
  construction from Example \ref{exa:3sec gud}, we realize the last
  real schemes $\langle 5 \sqcup   1  \langle b \rangle\rangle$ and
  $\langle b\sqcup   1  \langle 5 \rangle\rangle$ with
  $b=3,4$.  
\end{proof}

\begin{proof}[Proof of Theorem \ref{thm:sextics}]
  This is a consequence of Corollary \ref{cor:max tang 3-sec} and
  Theorem \ref{thm:bit}. Let $\C$ be a sextic obtained by 
  applying Proposition \ref{prop:conic2} to one of the 3-sections $C$
  whose existence is guaranteed by 
  Corollary \ref{cor:max tang 3-sec}. Thanks to  Theorem
  \ref{thm:bit},
 one easily counts the number of  real bitangents
 of $\C$.
 \begin{itemize}
 \item Any of the
   \[
   {{24}\choose{2}}=276
   \]
 real lines $L_{i,j}$  
 deforms  to a real bitangent of $\C$.
   \item Each zig-zag of the non-contractible component of $C$
     corresponds to two lines of type $L_i$.  One of them deforms to 
     two real bitangents of $\CC$, while the
     other one deforms to two complex conjugated
     bitangents.
     The number of such zig-zag is
     \[
     12-(p+n-1).
     \]
  \item    Each oval $\mathcal O$ of $\RR C$  corresponds to two lines of type
    $L_i$, which deform to 2 real (respectively complex
    conjugated) bitangents  of $\C$ if  $\mathcal O$
    deforms to an even (respectively odd) oval of $\RR\CC$.
 \end{itemize}
 Altogether, we see that the curve $\C$ has exactly
 \[
 276+ 2\times(12-(p+n-1)) +4(p-1)= 2(p-n)+298
 \]
 real bitangents.
\end{proof}

\subsection{Asymptotical constructions}
Here we prove  Theorem \ref{thm:asym1}(ii), and
Theorem \ref{thm:asym2}.
Both are consequences of an immediate combination of
Theorem \ref{thm:bit} and the next two propositions.
Recall that we choose a coordinate system $(u,v)$ on $\CC^2$.
\begin{prop}\label{prop:asymconstr1}
  There exists a family $(C_{k})_{k\ge 2}$ of non-singular real
  $k$-sections  in $\CC^2$ whose union with the line $v=0$ realizes the
  $\mathcal L$-scheme
   depicted in Figure \ref{fig:asymfam1}.
       \begin{figure}[h!]
\centering
\begin{tabular}{c}
  \includegraphics[width=15cm, angle=0]{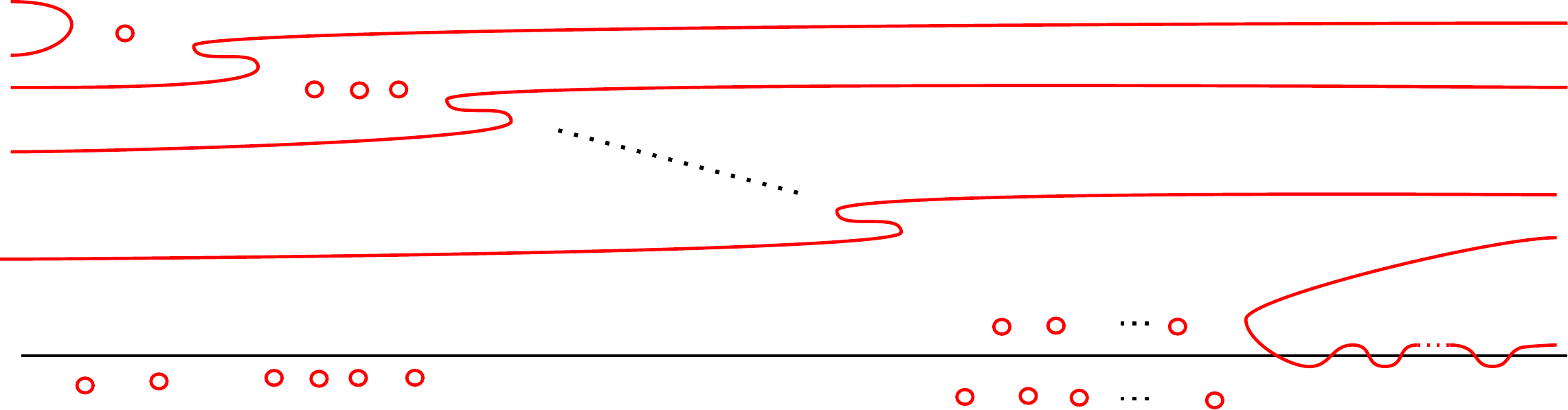}
   \put(-170,-5){$\underbrace{\qquad\qquad\qquad\quad}_{2k-2\mbox{ ovals}}$}
     \put(-80,-5){$\underbrace{\qquad\qquad\qquad}_{\substack{4k\mbox{ intersection}\\ \mbox{points}}}$}
  \put(-160,30){$\overbrace{\qquad\quad\quad\quad\ }^{2k-3\mbox{ ovals}}$}
  \\
  \\ 
\end{tabular}
\caption{The curve $\RR C_k$ (in red) and the line $v=0$ (in black).}
\label{fig:asymfam1}
    \end{figure}
\end{prop}
\begin{proof}
  The construction, by induction, is elementary and
  follows Harnack's method \cite{Har}.
  Let
  \[
  u_{1,1}<u_{1,2}<u_{1,3}<u_{1,4}<u_{2,1}<\cdots
  <u_{2,8}<u_{3,1}<\cdots<u_{3,12}<\cdots
  <u_{k,1}<\cdots<u_{k,4k}<\cdots
  \]
  be a sequence of real numbers, and let $(P_0)_{k\ge 1}$ be the
  sequence of polynomials defined by  induction as follows:
  \[
  P_0(u,v)=1
  \qquad\mbox{and}\qquad
   P_{k+1}(u,v)=vP_{k}(u,v)- \varepsilon_k \prod_{i=1}^{4k+4}(u-u_{k,i}).
   \]
   The curve $C_k$ whose existence is attested by the proposition is
   defined by the polynomial $P_k(u,v)$ for
   \[
   0<\varepsilon_k\ll \epsilon_{k-1}\ll\ldots\ll \varepsilon_1\ll 1.
   \]
   We show in Figure 
   \ref{fig:1sec} the construction for the steps $k=1$ and $k=2$.
\begin{figure}[h!]
\centering
\begin{tabular}{ccc}
  \includegraphics[width=6cm, angle=0]{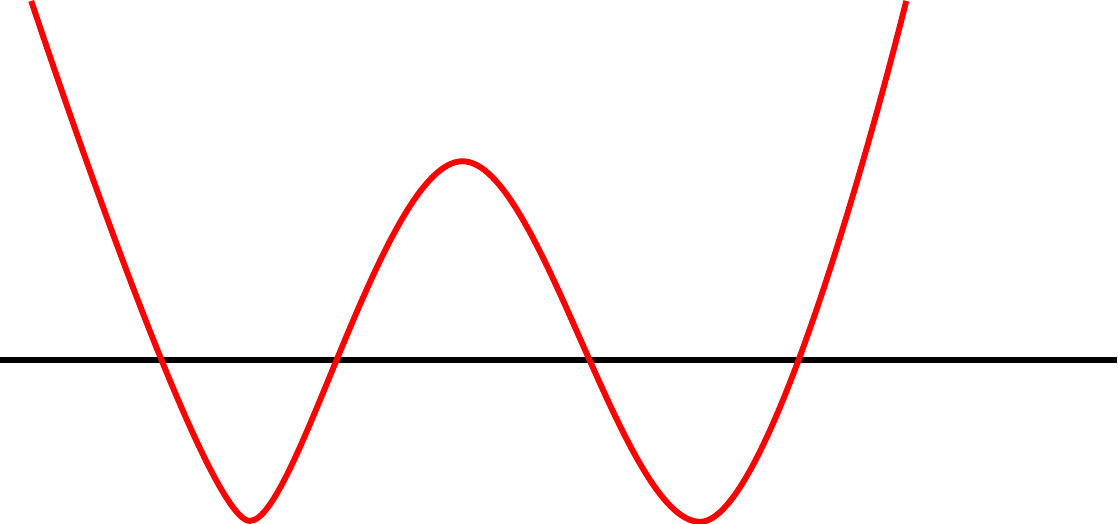}
  &\hspace{5ex}&
   \includegraphics[width=6cm, angle=0]{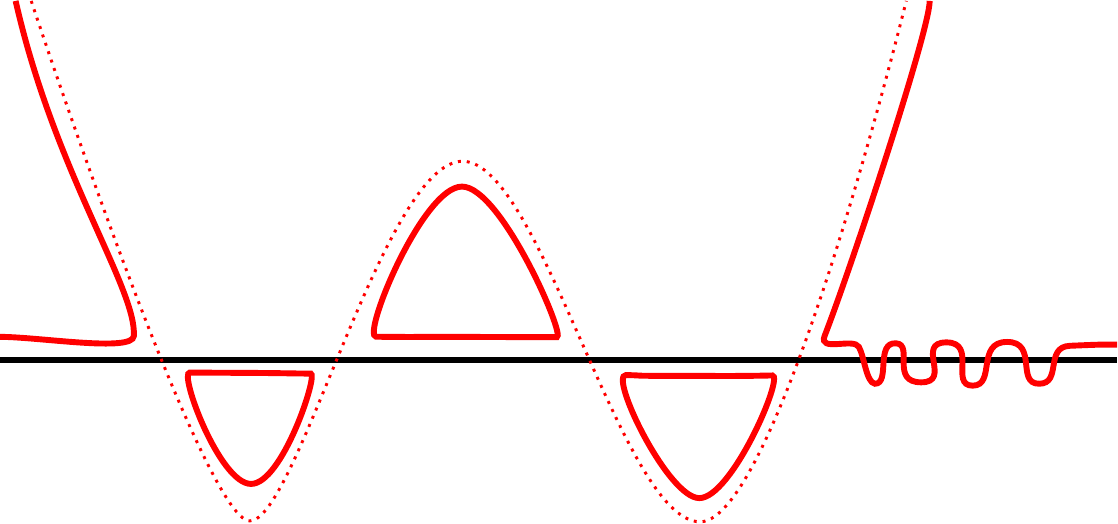}
  \\
  \\ $\RR C_1$ && $\RR C_2$
\end{tabular}
\caption{The curves  $\RR C_1$ and $\RR C_2$ from the proof of
  Proposition \ref{prop:asymconstr1}.}
\label{fig:1sec}
    \end{figure}
\end{proof}
\begin{remark}
  The $3$-section $C_3$ is used in Example \ref{ex:3sec}.  
\end{remark}

\begin{proof}[Proof of Theorem \ref{thm:asym1}(ii)]
  Denote by $\C_{k}$ the real algebraic cuvre of degree $2k$ in $\RP^2$ obtained by
  applying Proposition \ref{prop:conic2}, with $t<0$, to the $k$-section $C_k$ from
  Proposition \ref{prop:asymconstr1}. It is a Harnack curve whose real
  part in $\RP^2$, up to isotopy, is depicted in Figure
  \ref{fig:harnack}.
  Thanks to Theorem \ref{thm:bit}, the exact numbers of
  real and non-real bitangents of $\C_k$ are easy to compute.
  Since all  vertical tangent lines to
   $C_k$ are real,
  pairs of complex conjugated
  bitangents to $\C_k$ correspond to  pairs $(p,q)\in \RR\C_k^2$ where
  \begin{enumerate}
  \item $p$ is a tangency point of $C_k$ with a vertical line;
    \item $q$ has the same $u$-coordinate than $p$, but a lower $v$-coordinate.
  \end{enumerate}
  Hence the number of such pairs of complex conjugated
  bitangents is exactly
  \begin{align*}
    \sum_{i=1}^{k-2}4i(k-1-i)&=
    2(k-1)^2(k-2)  -\frac{4(k-1)(k-2)(2k-3)}6
    \\ &= \frac{2k(k-1)(k-2)}3.
  \end{align*}
  Multiplying by 2, we obtain the announced number of non real
  bitangents of $\C_k$.
\end{proof}

\begin{propo}\label{prop:asymconstr2}
  There exists a family $(C_{k})_{k\ge 1}$ of non-singular real
  $k$-sections in $\CC^2$ such that 
  \begin{itemize}
  \item the closure $\overline{\mathbb{R} C_{k}}$ of $\mathbb{R} C_{k}$ in
    $\mathcal O_{\CP^1}(4)$ is connected for any $k\ge 1$;
    
    \item the number $r_k$ of real vertical lines tangent to $C_k$  satisfies
      \[
      r_k=k^2+O(k).
      \]
  \end{itemize}
\end{propo}
\begin{proof}
  The proof uses again perturbations of  suitable reducible
  nodal curves, and  consists in two steps: first we treat the case of odd
  $k$'s, for which we construct a curve $\C_k$ with $r_k=4k(k-1)$
  (which is the maximal possible value);
  next we construct the curve $\C_k$ with $k$ even out
  of $\C_{k-1}$.

  \noindent {\bf Step 1: $k$ odd.} Consider the union $C_{k,red}$ of
  $k$ distinct 1-sections,
  with $\mathcal   L$-scheme
  depicted in Figure \ref{fig:red}.
       \begin{figure}[h!]
\centering
\begin{tabular}{c}
  \includegraphics[width=15cm, angle=0]{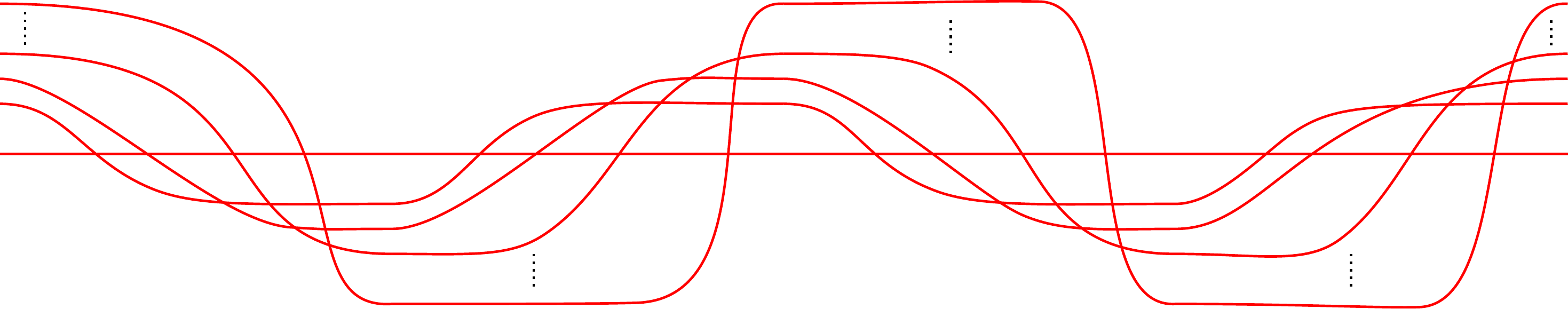}
  \\
  \\ 
\end{tabular}
\caption{A union of 1-sections.}
\label{fig:red}
    \end{figure}
       By Brusotti's Theorem \cite{DK,Mang17},
       there exists a non-singular $k$-section 
  $C_k$ obtained by perturbing $C_{k,red}$ in such a
  way that each double point  gives rise to two real  tangencies with
  a vertical line. One sees easily, for example by induction on
      $k$, that  $\overline{\RR \C_k}$
     is connected.

     \noindent {\bf Step 2: $k$ even.}
     Suppose that $k\ge 2$ is even, and consider the curve $C_{k-1}$
     constructed in Step 1. Choose any 1-section $\widetilde C_1$
     intersecting transversally 
  $C_{k-1}$ and such that
     $\RR C_{k-1}\cap \RR \widetilde C_1\ne \emptyset$. By Brusotti's
     Theorem again, one can perturb the reducible curve
     $ C_{k-1}\cup\widetilde C_1$ into a non-singular $k$-section
     $C_k$ such
     that  $\overline{\RR \C_k}$
     is connected. Such perturbation can be obtained by choosing the
     perturbation of the real nodes of $C_{k-1}\cup\widetilde C_1$
     one after the other, ensuring each time
     that the real part remains connected.
     By construction, the number of real vertical tangent to  $C_k$
     is at least the number of such real vertical tangents to the
      curve $C_{k-1}$.
\end{proof}

\begin{proof}[Proof of Theorem \ref{thm:asym2}]
     Denote by $\C_{k}$ the real algebraic curve of degree $2k$ in
   $\RP^2$ obtained by 
  applying Proposition \ref{prop:conic2}  to the $k$-section $C_k$ from
  Proposition \ref{prop:asymconstr2} (the sign of $t$ doesn't
  matter here). By Theorem \ref{thm:bit}, the curve $\C_k$ has at
  least
  ${r_k}\choose{2}$ real bitangents, which proves the theorem.
\end{proof}

\subsection{On the optimality of $t_s$}
The proof of Proposition \ref{prop:optimality} follows easily from
Theorem \ref{thm:bit} and next lemma.

\begin{lem}\label{lem:harn 4k+2}
  For any integer $k\ge 1$, there exists a $k$-section with the
  $\mathcal L$-scheme depicted in Figure \ref{fig:optimality}a or b
  depending on whether $k$ is odd or even, respectively.
 \begin{figure}[h!]
\centering
\begin{tabular}{ccc}
  \includegraphics[height=0.85cm, angle=0]{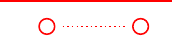}
   \put(-92,-5){$\smash{\underbrace{\qquad\qquad\qquad\quad
      }_{a\mbox{ ovals}}}$}
 &\hspace{5ex}&
 \includegraphics[height=0.85cm, angle=0]{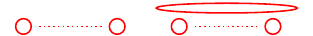}
   \put(-95,-5){$\smash{\underbrace{\qquad\qquad\qquad\quad
      }_{  b\mbox{ ovals}}}$}
   \put(-205,-5){$\smash{\underbrace{\qquad\qquad\qquad\quad
      }_{  a-b\mbox{ ovals}}}$}
 \\ \\
  \\a) $k$ odd && b) $k$ even, $b\le \frac{k(k-2)}{2}$
\end{tabular}
\caption{$0\le a\le \frac{3k(k-1)}{2}$.}
\label{fig:optimality}
    \end{figure}
\end{lem}
\begin{proof}
 Consider the rational  curve $C_{k,rat}$ with parameterization
       \[
       \begin{array}{cccc}
      \phi_k: &  \CC&\longrightarrow & \CC^2
         \\ &\theta &\longmapsto & (\theta^k,(\theta-1)^{4k})
       \end{array}.
       \]
       Let $P(X,Y)\in\RR[X,Y]$ be a real polynomial defining 
       $C_{k,rat}$. Without loss of generality, we may assume that
       $P(x,y)>0$ when $y$ is negative and large  enough in absolute
       value with respect to $x$.
       By\footnote{The parameterization of $C_{k,rat}$  is a rather
         simple  case of the general situation covered in
         \cite{Bru14b}, and
         one may
         prove directly  by elementary considerations
     the properties of  $C_{k,rat}$ that are needed here.}
       \cite[Theorem 10]{Bru14b}, the   curve $C_{k,rat}$ is a rational Harnack curve,
       and  has the following properties:
       \begin{itemize}
         \item  $C_{k,rat}$ has  $(2k-1)(k-1)$ nodes as its only singularities, which
           are all hyperbolic real nodes (i.e. image of non-real
           points of $\CC$);
       \item $\RR C_{k,rat}$ has $\frac{(k-1)(k-2)}{2}$ real nodes  located in the
         closure in $\RR^2$ of 
         $\{(x,y)\in\RR^2\ | \ P(x,y)<0\}$, and
          $\frac{3k(k-1)}{2}$ real nodes located in the
        closure of
        $\{(x,y)\in\RR^2\ | \ P(x,y)>0\}$; exactly
        $\frac{k(k-2)}2$ of these latter  have a positive
        $x$-coordinate if $k$ is even. 
       \end{itemize}
       One  checks easily that
       distinct  nodes of  $C_{k,rat}$ have distinct abscissa.
       Let 
       \[
       \begin{array}{cccc}
      \widetilde \phi_k: &  \CC\setminus R_2^{-1}(0)&\longrightarrow & \CC^2
         \\ &\theta &\longmapsto & (\frac{R_1(\theta)}{R_2(\theta)},\frac{(\theta-1)^{4k}}{R_2^4(\theta)})
       \end{array},
       \]
       be a perturbation of the map $\phi_k$ 
       with $R_1(\theta)$ and  $R_2(\theta)$ two  small perturbations
       of degree $k$ of the polynomials  $\theta^k$ and $1$, respectively, with only
       simple roots and $\floor{\frac k2}$ non-real roots.
       Then the real algebraic curve $\widetilde C_{k,rat}$ has the
       same above properties than the curve $C_{k,rat}$, and the
       additional property that
       the $1$-dimensional part of
       $\RR \widetilde C_{k,rat}$ has 0 or 2  vertical tangent lines
       depending on whether $k$
       is odd or even, respectively.

       Let $a\le \frac{3k(k-1)}2$.
By Brusotti's Theorem, 
there exists a non-singular real
$k$-section obtained as a perturbation of  $ \widetilde C_{k,rat}$
and such that
\begin{itemize}
\item exactly $a$ nodes of $\RR \widetilde C_{k,rat}$ located in the
        closure of
        $\{(x,y)\in\RR^2\ | \ P(x,y)>0\}$ deforms to an oval;
        \item all the other nodes of  $\RR \widetilde C_{k,rat}$ deforms to the
        empty set.
\end{itemize}
Hence the lemma is proved.
\end{proof}


\begin{proof}[Proof of Proposition \ref{prop:optimality}]
  It follows from Proposition \ref{prop:conic2} and Theorem
  \ref{thm:bit} that the real algebraic curves of degree $2k$
  constructed from the $k$-sections from Lemma \ref{lem:harn 4k+2},
  and $t>0$,
  has no negative real bitangents.
\end{proof}

\begin{remark}
  As noticed by Ilia Itenberg,
  one can prove Proposition \ref{prop:optimality}$(i)$
 without using Proposition
  \ref{prop:conic2} and Theorem 
  \ref{thm:bit}. Itenberg's proof has the same flavor than our, and we
  leave to  the reader  a few details to check.
  By \cite{KenOko06}, the rational  curve  in $\CP^2$ with parameterization
       \[
       \begin{array}{ccc}
        \CC P^1&\longrightarrow & \CC P^2
         \\ \left[\theta:\kappa\right] & \longmapsto &
            \left[\theta^{2k}:(\theta-\kappa)^{2k}: \kappa^{2k}\right]
       \end{array}
       \]
       is isotopic to the Figure \ref{fig:harnack} where all the empty
       ovals are contracted to points.
       Perturb the three  polynomials $\theta^{2k}$,
       $(\theta-\kappa)^{2k}$, and $\kappa^{2k}$ into three real
       polynomials in $\CC[\theta,\kappa]$
       of degree $2k$ with no real roots, and perturb the obtained
       rational curve into a non-singular real algebraic
       curve having exactly $p$
       empty ovals which are all convex.
       It is a well-known fact that
       two disjoint convex ovals in $\RP^2$ lying one outside the
       other have exactly $4$ real bitangents, see for example
       \cite[Theorem 3]{C94}.
       Combining this with Klein Formula \eqref{eq:klein} gives the result.
\end{remark}

\section{Perspectives}\label{sec:persp}
The results and techniques discussed in this article can be further
developed. We conclude by outlining some possible further research directions.

\begin{enumerate}
\item[(1)] Is the invariant $t_s$ sharp for the isotopy types depicted
  in Figure \ref{fig:sharp}?

\item[(2)] Determine $t_{max}(6)$.
  
\item[(3)] The value $t_{max}(6)\ge 318$ indicates that the lower
  bound on $t_{max}(d)$
  from Theorem \ref{thm:asym1}$(ii)$ can probably be improved. 

\item[(4)] What about bitangents of curves of odd degree in $\CP^2$? Does there
  exist any signed enumeration in this case? Can one locate  bitangents of
  a perturbation of a line and a multiple conic?
In the same vein, it could be interesting to generalize this work
  to
  singular curves, which together with their dual curve have cusps and nodes as their only singularities.

\item[(5)] One could also try to locate bitangents of a perturbation
  of a multiple curve of degree $l\ge 3$.
  For example, let $\C_0$ be a non-singular cubic in $\CP^2$, and $C$
   a general $k$-section of the normal bundle $N_{\C_0/\CP^2}$ that deforms
  to an algebraic curve $\C$ of degree $3k$ in $\CP^2$. Denote by
  $p_1,\ldots,p_{9k(k-1)}$ the points on $\C_0$ that correspond to
  fibers of $N_{\C_0/\CP^2}$ that are tangent to $C$.
Then, over $\CC$,  one can prove that:
  \begin{itemize}
  \item a line passing through two distinct $p_i$ deform to one
    bitangent of $\C$;
    \item a line tangent to $\C_0$ at $p_i$ deforms to $2(k-2)$
      bitangents of $\C$;
    \item
      a line tangent to $\C_0$ and passing through $p_i$ deforms to
      $k$ bitangents of $\C$ (there exist 4 of such lines for each
      $p_i$);
      \item a tangent line of $\C_0$ at an inflection point deforms to
        $3{{k}\choose{2}}$ bitangents of $\C$ (there exist 9 such
        inflexion points).
  \end{itemize}
  
\item[(6)] Using such degenerations to the normal cone of a cubic
  curve $\C_0$, one could also locate (at least partially) conics that are
  $5$-tangents to the perturbation of a multiple of $\C_0$. For
  example, any conic passing through a
  $5$-uple of distinct points $p_i$ as in item (5)
  deforms to a unique 5-tangent conic to $\C$. From this, one deduce
  easily the following.
  \begin{prop}\label{prop:5conic}
  There exists a sequence $(C_{k})_{d\ge 1}$ of real plane algebraic curves
of   degree $d=3k$  with
$\frac{d^{10}}5! +O(d^9)$ real 5-tangent conics
  (i.e. all 5-tangent conics of $C_k$ are asymptotically real).
  \end{prop}
  The number of $5$-tangent conics to a generic algebraic curve of
  degree $d$ in $\CP^2$ has been computed by Gathmann in \cite{Gat05}, to
  which we refer for more details about this problem.

\item[(7)] What about other fields? Is there a generalization of
  Theorem \ref{thm:signed} over any field $\mathbb K$ of
  characteristic 0, for
  compactly supported $\mathbb A^1$-Euler characteristic?
Furthermore analogously to Proposition \ref{prop:5conic}, 
  Theorem \ref{thm:bit} has an immediate weak version valid for any
  field $\mathbb K$ of characteristic not 2 (we use notations from Theorem
  \ref{thm:bit}): if $C$ is defined over $\mathbb K$,
  any line $L_{i,j}$ that is defined over $\mathbb K$ deforms to a bitangent
defined over $\mathbb K$ of $\C_t$, for $t\in\mathbb K$.
This provides a tool  to estimate the maximal number $t_{max,\mathbb K}(d)$ 
 of $\mathbb K$-bitangents that an algebraic curve
in $\mathbb KP^2$ may have.

\item[(8)] What about higher dimensions?
  Is there any generalization of Theorem \ref{thm:signed} to
   tritangent planes of real algebraic surfaces in $\CP^3$? To
  tritangent hyperplane sections of real algebraic curves on a real quadric in $\CP^3$?

\end{enumerate}

\bibliographystyle{plain}
\bibliography{biblio}

\end{document}